\documentclass[reqno,10pt,centertags]{amsart}
\usepackage{amsmath,amsthm,amscd,amssymb,latexsym,upref}
\usepackage{hyperref}

\newcommand{\bbN}{{\mathbb{N}}}
\newcommand{\bbR}{{\mathbb{R}}}

\newcommand{\bbC}{{\mathbb{C}}}

\newcommand{\cB}{{\mathcal B}}

\newcommand{\cD}{{\mathcal D}}

\newcommand{\cH}{{\mathcal H}}

\newcommand{\cS}{{\mathcal S}}

\newcommand{\cV}{{\mathcal V}}
\newcommand{\cX}{{\mathcal X}}
\newcommand{\cW}{{\mathcal W}}

\newcommand{\dott}{\,\cdot\,}
\newcommand{\no}{\notag}
\newcommand{\lb}{\label}
\newcommand{\f}{\frac}

\newcommand{\ol}{\overline}

\newcommand{\wti}{\widetilde}

\newcommand{\Oh}{O}

\newcommand{\dom}{\text{\rm{dom}}}

\newcommand{\supp}{\text{\rm{supp}}}

\newcommand{\bi}{\bibitem}
\newcommand{\hatt}{\widehat}
\newcommand{\beq}{\begin{equation}}
\newcommand{\eeq}{\end{equation}}
\newcommand{\ba}{\begin{align}}
\newcommand{\ea}{\end{align}}

\newcommand{\abs}[1]{\lvert#1\rvert}

\renewcommand{\Re}{\text{\rm Re}}

\renewcommand{\ge}{\geqslant}
\renewcommand{\le}{\leqslant}

\newcommand{\norm}[1]{\left\Vert#1\right\Vert}

\newcommand{\Om}{\Omega}
\newcommand{\dOm}{{\partial\Omega}}

\newcommand{\ga}{\gamma}

\newcommand{\LOm}{L^2(\Om;d^nx)}
\newcommand{\LdOm}{L^2(\dOm;d^{n-1} \omega)}

\allowdisplaybreaks \numberwithin{equation}{section}
\newtheorem{theorem}{Theorem}[section]

\newtheorem{lemma}[theorem]{Lemma}
\newtheorem{corollary}[theorem]{Corollary}
\newtheorem{hypothesis}[theorem]{Hypothesis}
\theoremstyle{definition}

\newtheorem{remark}[theorem]{Remark}

\newtheorem{example}[theorem]{Example}

\begin{document}

\title[Some Remarks on a Paper by Filonov]
{Nonlocal Robin Laplacians and Some Remarks on a Paper by Filonov 
 on Eigenvalue Inequalities}
\author[F.\ Gesztesy and M.\ Mitrea]{Fritz Gesztesy and Marius Mitrea}
\address{Department of Mathematics,
University of Missouri, Columbia, MO 65211, USA}
\email{gesztesyf@missouri.edu}
\urladdr{http://www.math.missouri.edu/personnel/faculty/gesztesyf.html}
\address{Department of Mathematics, University of
Missouri, Columbia, MO 65211, USA}
\email{mitream@missouri.edu}
\urladdr{http://www.math.missouri.edu/personnel/faculty/mitream.html}
\thanks{Based upon work partially supported by the US National Science
Foundation under Grant Nos.\ DMS-0400639 and FRG-0456306.}
\dedicatory{Dedicated with great pleasure to Sergio Albeverio on the occasion of his 70th birthday}
%\date{\today}
%\date{August 1, 2008.}
\thanks{{\it J. Diff. Eq.} {\bf 247}, 2871--2896 (2009).}

%%%%%%%%%%%%%%%%%%%%%%%%%%%%%%%%%%%%%%%%
%%%%%%%%%%%%%%%%%%%%%%%%%%%%%%%%%%%%%%%%
\begin{abstract}
The aim of this paper is twofold: First, we characterize an essentially optimal 
class of boundary operators $\Theta$ which give rise to self-adjoint  
Laplacians $-\Delta_{\Theta, \Om}$ in $L^2(\Om; d^n x)$ with (nonlocal and local) 
Robin-type boundary conditions on bounded Lipschitz domains $\Om\subset\bbR^n$, 
$n\in\bbN$, $n\geq 2$. Second, we extend Friedlander's inequalities between Neumann and Dirichlet Laplacian eigenvalues to those between nonlocal Robin and Dirichlet Laplacian eigenvalues associated with bounded Lipschitz domains $\Om$, following an approach introduced by Filonov for this type of problems. 
\end{abstract}
%%%%%%%%%%%%%%%%%%%%%%%%%%%%%%%%%%%%%%%%
%%%%%%%%%%%%%%%%%%%%%%%%%%%%%%%%%%%%%%%%

\maketitle

%%%%%%%%%%%%%%%%%%%%%%%%%%%%%%%%%%%%%%%%
%%%%%%%%%%%%%%%%%%%%%%%%%%%%%%%%%%%%%%%%
\section{Introduction}\label{s1}
%%%%%%%%%%%%%%%%%%%%%%%%%%%%%%%%%%%%%%%%
%%%%%%%%%%%%%%%%%%%%%%%%%%%%%%%%%%%%%%%%

In recent years, there has been a flurry of activity in connection with 2nd-order elliptic partial differential operators, particularly, Schr\"odinger--type operators on open domains 
$\Om\subset\bbR^n$, $n\in\bbN$, $n\geq 2$, with nonempty boundary 
$\partial\Om$, under various smoothness assumptions (resp., lack thereof) on 
$\Om$, and associated nonlocal Robin boundary conditions. We refer, for instance, 
to \cite{AM07}, \cite{AW03}, \cite{AW03a}, \cite{BD08}, \cite{DD97}, \cite{Da00}, 
\cite{Da06}, \cite{DK07}, \cite{DKM08}, \cite{GM08a}, \cite{GM09}, \cite{GS05}, 
\cite{GS08}, \cite{Ke09}, \cite{LS04}, \cite{LP08}, \cite{Wa06}, and the literature cited therein. 

If $\Om$ is minimally smooth, that is, a Lipschitz domain, these Robin-type boundary conditions are formally of the type 
\begin{equation}
\f{\partial u}{\partial\nu}\bigg|_\dOm + \Theta (u|_\dOm)=0   \lb{1.1}
\end{equation} 
in appropriate Sobolev spaces on the boundary $\dOm$, where $\nu$ denotes the outward pointing normal unit vector to $\partial\Om$, and $\Theta$ is an appropriate 
self-adjoint operator in $L^2(\partial\Omega;d^{n-1}\omega)$, with  $d^{n-1} \omega$ the surface measure on $\partial\Omega$. The boundary condition in 
\eqref{1.1} is called {\it local} and then resembles the familiar classical Robin boundary condition for smooth domains $\Om$, if $\Theta$ equals the operator of multiplication 
$M_{\theta}$ by an appropriate function $\theta$ on the boundary $\dOm$ (cf., e.g., \cite{Si78}). Otherwise, the  boundary condition \eqref{1.1} represents a {\it generalized} or 
{\it nonlocal} Robin boundary condition generated by the operator $\Theta$. The case $\Theta=0$ (resp., $\theta=0$), of course, corresponds to the case of Neumann boundary conditions on $\dOm$. The case of Dirichlet boundary conditions on $\dOm$, that is, the condition $u|_\dOm=0$ (formally corresponding to $\Theta=\infty$, resp., $\theta=\infty$) will also play a major role in this paper.   

Schr\"odinger operators on bounded Lipschitz domains $\Om$ with nonlocal Robin boundary conditions of the form \eqref{1.1}, have been very recently discussed in great detail in \cite{GM08a} and 
\cite{GM09}, and our treatment of nonlocal Robin Laplacians in this paper naturally builds upon these two papers. 

In addition to presenting a detailed approach to nonlocal Robin Laplacians on bounded Lipschitz domains, we also present an application to eigenvalue inequalities between the associated Robin and Dirichlet Laplacian eigenvalues, extending Friedlander's eigenvalue inequalities between Neumann and Dirichlet eigenvalues for bounded $C^1$-domains \cite{Fr91}, employing its extension to very general bounded domains due to Filonov \cite{Fi04}. We briefly review the relevant history of these eigenvalue inequalities. We denote by 
\begin{equation}\label{1.2}
0=\lambda_{N,\Om,1} < \lambda_{N,\Om,2}\leq\cdots\leq
\lambda_{N,\Om,j}\leq\lambda_{N,\Om,j+1}\leq\cdots
\end{equation}
the eigenvalues for the Neumann Laplacian $-\Delta_{N,\Om}$ in $L^2(\Om; d^n x)$, listed according to their multiplicity. Similarly,  
\begin{equation}\label{1.3}
0<\lambda_{D,\Om,1} < \lambda_{D,\Om,2}\leq\cdots\leq\lambda_{D,\Om,j}
\leq\lambda_{D,\Om,j+1}\leq\cdots
\end{equation}
denote the eigenvalues for the Dirichlet Laplacian $-\Delta_{D,\Om}$ in 
$L^2(\Om; d^n x)$, again enumerated according to their multiplicity.

Then, for any open bounded domain $\Om\subset\bbR^n$, the variational formulation of the Neumann and Dirichlet eigenvalue problem (in terms of Rayleigh quotients, cf.\ 
\cite[Sect.\ VI.1]{CH89}) immediately implies the inequalities
\begin{equation}
\lambda_{N,\Om,j} \leq \lambda_{D,\Om,j}, \quad j\in\bbN.   \lb{1.4}
\end{equation}
Moreover, P\'olya \cite{Po52} proved in 1952 that 
\begin{equation}
\lambda_{N,\Om,2} < \lambda_{D,\Om,1},   \lb{1.5}
\end{equation}
answering a question of Kornhauser and Stakgold \cite{KS52}. For a two-dimensional bounded convex domain $\Om \subset \bbR^2$, with a piecewise $C^2$-boundary $\dOm$, Payne \cite{Pa55} demonstrated in 1955 that
\begin{equation}
\lambda_{N,\Om,j+2} < \lambda_{D,\Om,j}, \quad j\in\bbN.   \lb{1.6}
\end{equation}
For domains $\Om$ with a $C^2$-boundary and $\dOm$ having a nonnegative mean curvature, Aviles \cite{Av86} showed in 1986 that 
\begin{equation}
\lambda_{N,\Om,j+1} < \lambda_{D,\Om,j}, \quad j\in\bbN.   \lb{1.7}
\end{equation}
This was reproved by Levine and Weinberger \cite{LW86} in 1986 who also showed that 
\begin{equation}
\lambda_{N,\Om,j+n} < \lambda_{D,\Om,j}, \quad j\in\bbN,   \lb{1.8}
\end{equation}
for smooth bounded convex domains $\Om$, as well as
\begin{equation}
\lambda_{N,\Om,j+n} \leq \lambda_{D,\Om,j}, \quad j\in\bbN,    \lb{1.9}
\end{equation}
for arbitrary bounded convex domains. In addition, they also proved inequalities of the type $\lambda_{N,\Om,j+m} < \lambda_{D,\Om,j}$, $j\in\bbN$, for all $1\leq m \leq n$ under appropriate assumptions on $\dOm$ in \cite{LW86} (see also \cite{Le88}). For additional eigenvalue inequalities we refer to Friedlander \cite{Fr92}, \cite{Fr00}.

In 1991, and most relevant to our paper, Friedlander \cite{Fr91} proved that actually 
\begin{equation}
\lambda_{N,\Om,j+1} \leq \lambda_{D,\Om,j}, \quad j\in\bbN,   \lb{1.10}
\end{equation}
for any bounded domain $\Om$ with a $C^1$-boundary $\dOm$. We also refer to Mazzeo \cite{Ma91} for an extension to certain smooth manifolds, and to 
Ashbaugh and Levine \cite{AL97} and Hsu and Wang \cite{HW01} for the case of subdomains of the $n$-dimensional sphere $S^n$ with a smooth boundary and nonnegative mean curvature. (For intriguing connections between these eigenvalue inequalities with the null variety of the Fourier transform of the characteristic function of the domain $\Om$, we also refer to \cite{BLP08}.) Finally, inequality \eqref{1.10} was extended to any open domain $\Om$ with finite volume, 
and with the embedding $H^1(\Om)\hookrightarrow L^2(\Om;d^n x)$ compact, by 
Filonov \cite{Fi04} in 2004, who also proved strict inequality in \eqref{1.10}, that is, 
\begin{equation}
\lambda_{N,\Om,j+1} < \lambda_{D,\Om,j}, \quad j\in\bbN.   \lb{1.11}
\end{equation}
We emphasize that Filonov's conditions on $\Omega$ are equivalent to 
$-\Delta_{N,\Om}$, defined as the unique self-adjoint operator associated with the Neumann sesquilinear form in $L^2(\Om; d^nx)$,
\begin{equation}
a_N(u,v) = \int_{\Om} d^n x \, \ol{(\nabla u)(x)} \cdot (\nabla v)(x), 
\quad u,v \in H^1(\Om),
\end{equation}
having a purely discrete spectrum, that is,
\begin{equation}
\sigma_{\rm ess}(-\Delta_{N,\Om}) = \emptyset 
\end{equation} 
(cf.\ also our discussion in Lemmas \ref{lB.1}, \ref{lB.2}), where 
$\sigma_{\rm ess}(\cdot)$ abbreviates the essential spectrum. 
While Friedlander used techniques based on the Dirichlet-to-Neumann map and an appropriate trial function argument, Filonov found an elementary new proof directly based on eigenvalue counting functions (and the same trial functions). Friedlander's result \eqref{1.10} was recently reconsidered by Arendt and Mazzeo \cite{AM07}, 
which in turn motivated our present investigation into an extension of Filonov's result 
\eqref{1.11} to nonlocal Robin Laplacians $-\Delta_{\Theta,\Om}$. In fact, if 
\begin{equation} \lb{1.12}
\lambda_{\Theta,\Om,1}\leq\lambda_{\Theta,\Om,2}\leq\cdots\leq
\lambda_{\Theta,\Om,j}\leq\lambda_{\Theta,\Om,j+1}\leq\cdots, 
\end{equation}
denote the eigenvalues of the nonlocal Robin Laplacian $-\Delta_{\Theta,\Om}$, counting multiplicity, we will prove that 
\begin{equation}\label{1.13}
\lambda_{\Theta,\Om,j+1}<\lambda_{D,\Om,j},\quad j\in\bbN, 
\end{equation}
assuming appropriate hypotheses on $\Theta$, including, for instance, 
\begin{equation}\lb{1.14}
\Theta\leq 0
\end{equation}
in the sense that $\langle f,\Theta f\rangle_{1/2} \leq 0$ for every
$f\in H^{1/2}(\partial\Omega)$. Here, $\langle \dott, \dott \rangle_{1/2}$ denotes the duality pairing between $H^{1/2}(\dOm)$ and 
$H^{-1/2}(\dOm)=\big(H^{1/2}(\dOm)\big)^*$. Filonov's result was recently generalized to the Heisenberg Laplacian on certain three-dimensional domains by Hansson \cite{Ha08}. 

Most recently, the relation between the eigenvalue counting functions of the Dirichlet and Neumann Laplacian originally established by Friedlander in \cite{Fr91}, was discussed in an abstract setting by Safarov \cite{Sa08} based on sequilinear forms and an abstract version of the Dirichlet-to-Neumann map. When applied to elliptic boundary value problems, his approach avoids the use of boundary trace operators and hence is not plagued by the usual regularity hypotheses on the boundary (such as Lipschitz boundaries or additional smoothness of the boundary). In particular, Safarov's approach permits the existence of an essential spectrum of the Neumann (resp., Robin) and Dirichlet Laplacians and then restricts the eigenvalue inequalities of the type \eqref{1.10} to those Dirichlet eigenvalues lying strictly beyond 
$\inf\big(\sigma_{\rm ess}(-\Delta_{\Theta,\Om})\big)$. 
Hence, Safarov's results appear to be in the nature of best possible in this context. In addition, as pointed out at the end in Remark \ref{r5.4}, Safarov's novel approach considerably improves upon conditions such as \eqref{1.14}. 

Condition \eqref{1.14} was anticipated by Filonov in the special case of local Robin Laplacians $-\Delta_{M_\theta,\Om}$, where 
$M_\theta$ equals the operator of multiplication by an appropriate real-valued function 
$\theta$ on the boundary $\dOm$. The case of local Robin Laplacians 
$-\Delta_{M_\theta,\Om}$ associated with $C^{2,\alpha}$-domains $\Om\subset\bbR^n$, $\alpha \in (0,1]$, was discussed by Levine \cite{Le88} in 1988. Assuming $(n-1)h(\xi) \geq  \theta(\xi)$, $\xi\in\partial\Om$, $h(\cdot)$ the mean curvature on $\partial\Om$, he established 
\begin{equation}
\lambda_{M_\theta,\Om,j+1} < \lambda_{D,\Om,j}, \quad j\in\bbN.   \lb{1.15}
\end{equation}
He also proved 
\begin{equation}\lb{1.16}
\lambda_{M_\theta,\Om,j+n}<\lambda_{D,\Om,j},\quad j\in\bbN, 
\end{equation}
under the additional assumption of convexity of $\Om$. (In addition, he derived  inequalities of the type $\lambda_{\Theta,\Om,j+m}<\lambda_{D,\Om,j}$, $j\in\bbN$, for all $1\leq m \leq n$, under appropriate conditions on $\Om$.) Similarly, in the case of local Robin Laplacians $-\Delta_{M_\theta,\Om}$ on smooth domains $\Om\subset S^n$ and $(n-1)h(\xi) \geq  \theta(\xi)$, $\xi\in\partial\Om$, Ashbaugh and Levine \cite{AL97} proved $\lambda_{M_\theta,\Om,j+1} \leq \lambda_{D,\Om,j}$, $j\in\bbN$, in 1997.

We conclude this introduction with a brief description of the content of each section:  Section \ref{s2} succinctly reviews the basic facts on sesquilinear forms and their associated self-adjoint operators. Sobolev spaces on bounded Lipschitz domains and on their boundaries are presented in a nutshell in Section \ref{s3}. Section \ref{s4} focuses on self-adjoint realizations of Laplacians with nonlocal Robin boundary conditions, and finally, Section \ref{s5} discusses the extension of Friedlander's eigenvalue inequalities between Neumann and Dirichlet eigenvalues to that of nonlocal Robin eigenvalues and Dirichlet eigenvalues for bounded Lipschitz domains, closely following a strategy of proof due to Filonov.

%%%%%%%%%%%%%%%%%%%%%%%%%%%%%%%%%%%%%%
%%%%%%%%%%%%%%%%%%%%%%%%%%%%%%%%%%%%%%
\section{Sesquilinear Forms and Associated Operators}
\lb{s2}
%%%%%%%%%%%%%%%%%%%%%%%%%%%%%%%%%%%%%%
%%%%%%%%%%%%%%%%%%%%%%%%%%%%%%%%%%%%%%

In this section we describe a few basic facts on sesquilinear forms and
linear operators associated with them.
Let $\cH$ be a complex separable Hilbert space with scalar product
$(\dott,\dott)_{\cH}$ (antilinear in the first and linear in the second
argument), $\cV$ a reflexive Banach space continuously and densely embedded
into $\cH$. Then also $\cH$ embeds continuously and densely into $\cV^*$.
That is,
\begin{equation}
\cV  \hookrightarrow \cH  \hookrightarrow \cV^*.     \lb{B.1}
\end{equation}
Here the continuous embedding $\cH\hookrightarrow \cV^*$ is accomplished via
the identification
\begin{equation}
\cH \ni v \mapsto (\dott,v)_{\cH} \in \cV^*,     \lb{B.2}
\end{equation}
and we use the convention in this manuscript that if $X$ denotes a Banach space, 
$X^*$ denotes the {\it adjoint space} of continuous  conjugate linear functionals on $X$,
also known as the {\it conjugate dual} of $X$.

In particular, if the sesquilinear form
\begin{equation}
{}_{\cV}\langle \dott, \dott \rangle_{\cV^*} \colon \cV \times \cV^* \to \bbC
\end{equation}
denotes the duality pairing between $\cV$ and $\cV^*$, then
\begin{equation}
{}_{\cV}\langle u,v\rangle_{\cV^*} = (u,v)_{\cH}, \quad u\in\cV, \;
v\in\cH\hookrightarrow\cV^*,   \lb{B.3}
\end{equation}
that is, the $\cV, \cV^*$ pairing
${}_{\cV}\langle \dott,\dott \rangle_{\cV^*}$ is compatible with the
scalar product $(\dott,\dott)_{\cH}$ in $\cH$.

Let $T \in\cB(\cV,\cV^*)$. Since $\cV$ is reflexive, $(\cV^*)^* = \cV$, one has
\begin{equation}
T \colon \cV \to \cV^*, \quad  T^* \colon \cV \to \cV^*   \lb{B.4}
\end{equation}
and
\begin{equation}
{}_{\cV}\langle u, Tv \rangle_{\cV^*}
= {}_{\cV^*}\langle T^* u, v\rangle_{(\cV^*)^*}
= {}_{\cV^*}\langle T^* u, v \rangle_{\cV}
= \ol{{}_{\cV}\langle v, T^* u \rangle_{\cV^*}}.
\end{equation}
{\it Self-adjointness} of $T$ is then defined by $T=T^*$, that is,
\begin{equation}
{}_{\cV}\langle u,T v \rangle_{\cV^*}
= {}_{\cV^*}\langle T u, v \rangle_{\cV}
= \ol{{}_{\cV}\langle v, T u \rangle_{\cV^*}}, \quad u, v \in \cV,    \lb{B.5}
\end{equation}
{\it nonnegativity} of $T$ is defined by
\begin{equation}
{}_{\cV}\langle u, T u \rangle_{\cV^*} \geq 0, \quad u \in \cV,    \lb{B.6}
\end{equation}
and {\it boundedness from below of $T$ by $c_T \in\bbR$} is defined by
\begin{equation}
{}_{\cV}\langle u, T u \rangle_{\cV^*} \geq c_T \|u\|^2_{\cH},
\quad u \in \cV.
\lb{B.6a}
\end{equation}
(By \eqref{B.3}, this is equivalent to
${}_{\cV}\langle u, T u \rangle_{\cV^*} \geq c_T \,
{}_{\cV}\langle u, u \rangle_{\cV^*}$, $u \in \cV$.)

Next, let the sesquilinear form $a(\dott,\dott)\colon\cV \times \cV \to \bbC$
(antilinear in the first and linear in the second argument) be
{\it $\cV$-bounded}, that is, there exists a $c_a>0$ such that
\begin{equation}
|a(u,v)| \le c_a \|u\|_{\cV} \|v\|_{\cV},  \quad u, v \in \cV.
\end{equation}
Then $\wti A$ defined by
\begin{equation}
\wti A \colon \begin{cases} \cV \to \cV^*, \\
\, v \mapsto \wti A v = a(\dott,v), \end{cases}    \lb{B.7}
\end{equation}
satisfies
\begin{equation}
\wti A \in\cB(\cV,\cV^*) \, \text{ and } \,
{}_{\cV}\big\langle u, \wti A v \big\rangle_{\cV^*}
= a(u,v), \quad  u, v \in \cV.    \lb{B.8}
\end{equation}
Assuming further that $a(\dott,\dott)$ is {\it symmetric}, that is,
\begin{equation}
a(u,v) = \ol{a(v,u)},  \quad u,v\in \cV,    \lb{B.9}
\end{equation}
and that $a$ is {\it $\cV$-coercive}, that is, there exists a constant
$C_0>0$ such that
\begin{equation}
a(u,u)  \geq C_0 \|u\|^2_{\cV}, \quad u\in\cV,    \lb{B.10}
\end{equation}
respectively, then,
\begin{equation}
\wti A \colon \cV \to \cV^* \, \text{ is bounded, self-adjoint, and boundedly
invertible.}    \lb{B.11}
\end{equation}
Moreover, denoting by $A$ the part of $\wti A$ in $\cH$ defined by
\begin{align}
\dom(A) = \big\{u\in\cV \,|\, \wti A u \in \cH \big\} \subseteq \cH, \quad
A= \wti A\big|_{\dom(A)}\colon \dom(A) \to \cH,   \lb{B.12}
\end{align}
then $A$ is a (possibly unbounded) self-adjoint operator in $\cH$ satisfying
\begin{align}
& A \geq C_0 I_{\cH},   \lb{B.13}  \\
& \dom\big(A^{1/2}\big) = \cV.  \lb{B.14}
\end{align}
In particular,
\begin{equation}
A^{-1} \in\cB(\cH).   \lb{B.15}
\end{equation}
The facts \eqref{B.1}--\eqref{B.15} are a consequence of the Lax--Milgram
theorem and the second representation theorem for symmetric sesquilinear forms.
Details can be found, for instance, in \cite[Sects.\ VI.3, VII.1]{DL00},
\cite[Ch.\ IV]{EE89}, and \cite{Li62}.

Next, consider a symmetric form $b(\dott,\dott)\colon \cV\times\cV\to\bbC$
and assume that $b$ is {\it bounded from below by $c_b\in\bbR$}, that is,
\begin{equation}
b(u,u) \geq c_b \|u\|_{\cH}^2, \quad u\in\cV.  \lb{B.19}
\end{equation}
Introducing the scalar product
$(\dott,\dott)_{\cV_b}\colon \cV\times\cV\to\bbC$
(and the associated norm $\|\cdot\|_{\cV_b}$) by
\begin{equation}
(u,v)_{\cV_b} = b(u,v) + (1- c_b)(u,v)_{\cH}, \quad u,v\in\cV,  \lb{B.20}
\end{equation}
turns $\cV$ into a pre-Hilbert space $(\cV; (\dott,\dott)_{\cV_b})$,
which we denote by
$\cV_b$. The form $b$ is called {\it closed} in $\cH$ if $\cV_b$ is actually
complete, and hence a Hilbert space. The form $b$ is called {\it closable}
in $\cH$ if it has a closed extension. If $b$ is closed in $\cH$, then
\begin{equation}
|b(u,v) + (1- c_b)(u,v)_{\cH}| \le \|u\|_{\cV_b} \|v\|_{\cV_b},
\quad u,v\in \cV,
\lb{B.21}
\end{equation}
and
\begin{equation}
|b(u,u) + (1 - c_b)\|u\|_{\cH}^2| = \|u\|_{\cV_b}^2, \quad u \in \cV,
\lb{B.22}
\end{equation}
show that the form $b(\dott,\dott)+(1 - c_b)(\dott,\dott)_{\cH}$ is a
symmetric, $\cV$-bounded, and $\cV$-coercive sesquilinear form. Hence,
by \eqref{B.7} and \eqref{B.8}, there exists a linear map
\begin{equation}
\wti B_{c_b} \colon \begin{cases} \cV_b \to \cV_b^*, \\
\hspace*{.51cm}
v \mapsto \wti B_{c_b} v = b(\dott,v) +(1 - c_b)(\dott,v)_{\cH},
\end{cases}
\lb{B.23}
\end{equation}
with
\begin{equation}
\wti B_{c_b} \in\cB(\cV_b,\cV_b^*) \, \text{ and } \,
{}_{\cV_b}\big\langle u, \wti B_{c_b} v \big\rangle_{\cV_b^*}
= b(u,v)+(1 -c_b)(u,v)_{\cH}, \quad  u, v \in \cV.    \lb{B.24}
\end{equation}
Introducing the linear map
\begin{equation}
\wti B = \wti B_{c_b} + (c_b - 1)\wti I \colon \cV_b\to\cV_b^*,
\lb{B.24a}
\end{equation}
where $\wti I\colon \cV_b\hookrightarrow\cV_b^*$ denotes
the continuous inclusion (embedding) map of $\cV_b$ into $\cV_b^*$, one
obtains a self-adjoint operator $B$ in $\cH$ by restricting $\wti B$ to $\cH$,
\begin{align}
\dom(B) = \big\{u\in\cV \,\big|\, \wti B u \in \cH \big\} \subseteq \cH, \quad
B= \wti B\big|_{\dom(B)}\colon \dom(B) \to \cH,   \lb{B.25}
\end{align}
satisfying the following properties:
\begin{align}
& B \geq c_b I_{\cH},  \lb{B.26} \\
& \dom\big(|B|^{1/2}\big) = \dom\big((B - c_bI_{\cH})^{1/2}\big)
= \cV,  \lb{B.27} \\
& b(u,v) = \big(|B|^{1/2}u, U_B |B|^{1/2}v\big)_{\cH}    \lb{B.28b} \\
& \hspace*{.97cm}
= \big((B - c_bI_{\cH})^{1/2}u, (B - c_bI_{\cH})^{1/2}v\big)_{\cH}
+ c_b (u, v)_{\cH}
\lb{B.28} \\
& \hspace*{.97cm}
= {}_{\cV_b}\big\langle u, \wti B v \big\rangle_{\cV_b^*},
\quad u, v \in \cV, \lb{B.28a} \\
& b(u,v) = (u, Bv)_{\cH}, \quad  u\in \cV, \; v \in\dom(B),  \lb{B.29} \\
& \dom(B) = \{v\in\cV\,|\, \text{there exists an $f_v\in\cH$ such that}  \no \\
& \hspace*{3.05cm} b(w,v)=(w,f_v)_{\cH} \text{ for all $w\in\cV$}\},
\lb{B.30} \\
& Bu = f_u, \quad u\in\dom(B),  \no \\
& \dom(B) \text{ is dense in $\cH$ and in $\cV_b$}.  \lb{B.31}
\end{align}
Properties \eqref{B.30} and \eqref{B.31} uniquely determine $B$.
Here $U_B$ in \eqref{B.28} is the partial isometry in the polar
decomposition of $B$, that is,
\begin{equation}
B=U_B |B|, \quad  |B|=(B^*B)^{1/2} \geq 0.   \lb{B.32}
\end{equation}
The operator $B$ is called the {\it operator associated with the form $b$}.

The norm in the Hilbert space $\cV_b^*$ is given by
\begin{equation}
\|\ell\|_{\cV_b^*}
= \sup \{|{}_{\cV_b}\langle u, \ell \rangle_{\cV_b^*}| \,|\,
\|u\|_{\cV_b} \le 1\}, \quad \ell \in \cV_b^*,   \lb{B.34}
\end{equation}
with associated scalar product,
\begin{equation}
(\ell_1,\ell_2)_{\cV_b^*}
= {}_{\cV_b}\big\langle
\big(\wti B+(1-c_b)\wti I\,\big)^{-1}
\ell_1, \ell_2 \big\rangle_{\cV_b^*},
\quad \ell_1, \ell_2 \in \cV_b^*.   \lb{B.35}
\end{equation}
Since
\begin{equation}
\big\|\big(\wti B + (1 - c_b)\wti I\,\big)v \big\|_{\cV_b^*}
= \|v\|_{\cV_b}, \quad v\in\cV,   \lb{B.36}
\end{equation}
the Riesz representation theorem yields
\begin{equation}
\big(\wti B+(1-c_b)\wti I\,\big)\in\cB(\cV_b,\cV_b^*) \,
\text{ and }\big(\wti B + (1 - c_b)\wti I\,\big) \colon \cV_b
\to \cV_b^* \, \text{ is unitary.}   \lb{B.37}
\end{equation}
In addition,
\begin{align}
\begin{split}
{}_{\cV_b}\big\langle u,\big(\wti B
+ (1 - c_b)\wti I\,\big) v \big\rangle_{\cV_b^*}
& = \big(\big(B+(1-c_b)I_{\cH}\big)^{1/2}u,
\big(B+(1-c_b)I_{\cH}\big)^{1/2}v \big)_{\cH}  \\
& = (u,v)_{\cV_b},  \quad  u, v \in \cV_b.   \lb{B.38}
\end{split}
\end{align}
In particular,
\begin{equation}
\big\|(B+(1-c_b)I_{\cH})^{1/2}u\big\|_{\cH} = \|u\|_{\cV_b},
\quad u \in \cV_b,   \lb{B.39}
\end{equation}
and hence
\begin{equation}
(B+(1-c_b)I_{\cH})^{1/2}\in\cB(\cV_b,\cH) \, \text{ and }
(B + (1 - c_b)I_{\cH})^{1/2} \colon \cV_b \to \cH \, \text{ is unitary.}
\lb{B.40}
\end{equation}
The facts \eqref{B.19}--\eqref{B.40} comprise the second representation
theorem of sesquilinear forms (cf.\ \cite[Sect.\ IV.2]{EE89},
\cite[Sects.\ 1.2--1.5]{Fa75}, and \cite[Sect.\ VI.2.6]{Ka80}).

A special but important case of nonnegative closed forms is obtained as
follows: Let $\cH_j$, $j=1,2$, be complex separable Hilbert spaces, and
$T\colon \dom(T)\to\cH_2$, $\dom(T)\subseteq \cH_1$, a densely defined
operator. Consider the nonnegative form
$a_T\colon \dom(T)\times \dom(T)\to\bbC$ defined by
\begin{equation}
a_T(u,v)=(Tu,Tv)_{\cH_2}, \quad u, v \in\dom(T).   \lb{B.42}
\end{equation}
Then the form $a_T$ is closed (resp., closable) in $\cH_1$ if and only if $T$ is.
If $T$ is closed, the unique nonnegative self-adjoint operator associated
with $a_T$ in $\cH_1$, whose existence is guaranteed by the second
representation theorem for forms, then equals $T^*T\geq 0$. In particular,
one obtains in addition to \eqref{B.42}, 
\begin{equation}
a_T(u,v) = (|T|u,|T|v)_{\cH_1}, \quad u, v \in\dom(T)=\dom(|T|).   \lb{B.43}
\end{equation}
Moreover, since
\begin{align}
\begin{split}
& b(u,v) +(1-c_b)(u,v)_{\cH}
= \big((B+(1-c_b)I_{\cH})^{1/2}u, (B+(1-c_b)I_{\cH})^{1/2}v\big)_{\cH}, \\
& \hspace*{6.1cm}  u, v \in \dom(b) = \dom\big(|B|^{1/2}\big)=\cV,   \lb{B.43a}
\end{split}
\end{align}
and $(B+(1-c_b)I_{\cH})^{1/2}$ is self-adjoint (and hence closed) in
$\cH$, a symmetric, $\cV$-bounded, and $\cV$-coercive form is densely
defined in $\cH\times\cH$ and closed in $\cH$ (a fact we will be using in the proof of
Theorem \ref{t2.3}). We refer to \cite[Sect.\ VI.2.4]{Ka80} and
\cite[Sect.\ 5.5]{We80} for details.

Next we recall that if $a_j$ are sesquilinear forms defined on
$\dom(a_j)$, $j=1,2$, bounded from below and closed, then also
\begin{equation}
(a_1+a_2)\colon \begin{cases} (\dom(a_1)\cap\dom(a_2))\times
(\dom(a_1)\cap\dom(a_2)) \to \bbC, \\
(u,v) \mapsto (a_1+a_2)(u,v) = a_1(u,v) + a_2(u,v) \end{cases}   \lb{B.44}
\end{equation}
is bounded from below and closed (cf., e.g., \cite[Sect.\ VI.1.6]{Ka80}).

Finally, we also recall the following perturbation theoretic fact:
Suppose $a$ is a sesquilinear form defined on $\cV\times\cV$, bounded
from below and closed, and let $b$ be a symmetric sesquilinear form
bounded with respect to $a$ with bound less than one, that is,
$\dom(b)\supseteq \cV\times\cV$, and that there exist $0\le \alpha < 1$ and
$\beta\ge 0$ such that
\begin{equation}
|b(u,u)| \le \alpha |a(u,u)| + \beta \|u\|_{\cH}^2, \quad u\in \cV.   \lb{B.45}
\end{equation}
Then
\begin{equation}
(a+b)\colon \begin{cases} \cV\times\cV \to \bbC, \\
\hspace*{.12cm} (u,v) \mapsto (a+b)(u,v) = a(u,v) + b(u,v) \end{cases}
\lb{B.46}
\end{equation}
defines a sesquilinear form that is bounded from below and closed
(cf., e.g., \cite[Sect.\,VI.1.6]{Ka80}). In the special case where $\alpha$
can be chosen arbitrarily small, the form $b$ is called infinitesimally
form bounded with respect to $a$.

Finally we turn to a brief discussion of operators with purely discrete spectra. We denote by $\# S$ the cardinality of the set $S$.

%%%%%%%%%%%%%%%%%%%%%%%%%%%%%%%%%%%%%%
\begin{lemma}\label{lB.1}
Let $\cV$, $\cH$ be as in \eqref{B.1}, \eqref{B.2}.
Assume that the inclusion $\iota_{\cV}:\cV\hookrightarrow\cH$ is compact, and
that the sesquilinear form $a(\cdot,\cdot):\cV\times\cV\to\bbC$ is symmetric, 
$\cV$-bounded, and suppose that there exists $\kappa>0$ with the property that
\begin{equation}\label{B.50}
a_{\kappa}(u,v):=a(u,v)+\kappa\,(u,v)_{\cH},\quad u,v\in\cV,
\end{equation}
is $\cV$-coercive. Then the operator $A$ associated with
$a(\cdot,\cdot)$ is self-adjoint and bounded from below. In addition,
$A$ has purely discrete spectrum 
\begin{equation}
\sigma_{\rm ess} (A) = \emptyset,     \lb{B.51}
\end{equation}
and hence $\sigma(A)=\{\lambda_j(A)\}_{j\in\bbN}$, with $\lambda_j(A)\to\infty$ 
as $j\to\infty$,  
\begin{equation}\label{B.52}
-\kappa<\lambda_1(A)\leq\lambda_2(A)\leq\cdots\leq\lambda_j(A)
\leq\lambda_{j+1}(A)\leq\cdots \, .
\end{equation}
Here, the eigenvalues $\lambda_j(A)$ of $A$ are listed according to their multiplicity. 
Moreover, the following min-max principle holds:
\begin{equation}\label{B.53}
\lambda_j(A)=\min_{\stackrel{L_j\text{ subspace of }\cV}{\dim(L_j)=j}}
\Big(\max_{0\not=u\in L_j}R_a[u]\Big),\quad j\in\bbN,
\end{equation}
where $R_a[u]$ denotes the Rayleigh quotient
\begin{equation}\label{B.54}
R_a[u]:=\frac{a(u,u)}{\|u\|^2_{\cH}},\quad 0\not=u\in\cV.
\end{equation}
As a consequence, if $N_A$ is the eigenvalue counting function of $A$, that is,
\begin{equation}\label{B.55}
N_A(\lambda):=\#\,\{j\in\bbN\,|\,\lambda_j(A)\leq\lambda\},\quad\lambda\in\bbR,
\end{equation}
then for each $\lambda\in\bbR$ one has
\begin{equation}\label{B.56}
N_A(\lambda)=\max \big\{\dim (L)\in\bbN_0 \,\big|\, 
L \text{ a subspace of }  \cV \text{ with }
a(u,u)\leq\lambda\|u\|^2_{\cH}, \, u\in L\big\}.
\end{equation}
\end{lemma}
%%%%%%%%%%%%%%%%%%%%%%%%%%%
\begin{proof}
Analogous claims for the operator $B$ associated with the $\cV$-coercive
form $a_{\kappa}(\cdot,\cdot)$ are well-known (cf., e.g., 
\cite[Sect.\ VI.3.2.5, Ch.\ VII]{DL00}).
Then the corresponding claims for $A$ follow from these, after observing that
$B=A+\kappa I_{\cH}$, $R_{a_\kappa}[u]=R_a[u]+\kappa$,
$\lambda_j(B)=\lambda_j(A)+\kappa$, and $N_B(\lambda)=N_A(\lambda-\kappa)$, 
$\lambda\in\bbR$.
\end{proof}
%%%%%%%%%%%%%%%%%%%%%%%%%%%

A closely related result is provided by the following elementary observations: Let 
$c\in\bbR$ and $B\geq c I_{\cH}$ be a self-adjoint operator in $\cH$, and introduce the sesquilinear form $b$ in $\cH$ associated with $B$ via
\begin{align}
\begin{split}
& b(u,v) = \big((B - c I_{\cH})^{1/2} u, (B - c I_{\cH})^{1/2} v\big)_{\cH}
+ c (u,v)_{\cH}, \\  
& u,v \in \dom(b) = \dom\big(|B|^{1/2}\big).    \lb{B.57}
\end{split}
\end{align}
Given $B$ and $b$, one introduces the Hilbert space $\cH_b \subseteq \cH$ by 
\begin{align}
& \cH_b =\big(\dom\big(|B|^{1/2}\big), (\cdot,\cdot)_{\cH_b}\big),   \no \\ 
& (u,v)_{\cH_b} =  b(u,v) + (1-c) (u,v)_{\cH}   \lb{B.58} \\
& \hspace*{1.2cm} = \big((B - c I_{\cH})^{1/2} u, (B - c I_{\cH})^{1/2} v\big)_{\cH} 
+ (u,v)_{\cH}    \no \\
& \hspace*{1.2cm} = \big((B + (1-c) I_{\cH})^{1/2} u, (B + (1-c) I_{\cH})^{1/2} v\big)_{\cH}. 
\no 
\end{align}
Of course, $\cH_b$ plays a role analogous to $\cV_b$ in \eqref{B.20}. As in \eqref{B.40} one then observes that 
\begin{equation} 
(B + (1 - c)I_{\cH})^{1/2} \colon \cH_b \to \cH \, \text{ is unitary.}     \lb{B.59}
\end{equation}

%%%%%%%%%%%%%%%%%%%%
\begin{lemma} \lb{lB.2} 
Let $\cH$, $B$, $b$, and $\cH_b$ be as in \eqref{B.57}--\eqref{B.59}. Then $B$ has purely discrete spectrum, that is, $\sigma_{\rm ess} (B) = \emptyset$, 
if and only if $\cH_b \hookrightarrow \cH$ compactly. 
\end{lemma}
%%%%%%%%%%%%%%%%%%%%
\begin{proof}
Denoting by $J_{\cH_b}=I_{\cH}|_{\cH_b}$ the inclusion map from $\cH_b$ into 
$\cH$, one infers that
\begin{equation}
\cH \stackrel{(B+(1-c)I_{\cH})^{-1/2}}{-\!\!\!-\!\!\!-\!\!\!-\!\!\!-\!\!\!-\!\!\!-\!\!\!-\!\!\!-\!\!\!-\!\!\!-\!\!\!\longrightarrow} \cH_b 
\stackrel{J_{\cH_b}}{\hookrightarrow} \cH.
\end{equation}
Thus, one concludes that 
\begin{align}
& \cH_b \hookrightarrow \cH\, \text{compactly}  \iff J_{\cH_b} \in \cB_{\infty}(\cH_b,\cH) 
\no \\
& \quad \iff \big[J_{\cH_b}(B+(1-c)I_{\cH})^{-1/2}\big] (B+(1-c)I_{\cH})^{1/2} 
\in \cB_{\infty}(\cH_b,\cH)   \no \\
& \quad \iff J_{\cH_b}(B+(1-c)I_{\cH})^{-1/2} \in \cB_{\infty}(\cH) 
\iff (B+(1-c)I_{\cH})^{-1/2} \in \cB_{\infty}(\cH)   \no \\ 
& \quad \iff (B - zI_{\cH})^{-1} \in \cB_{\infty}(\cH), \quad  
z\in\bbC\backslash{\sigma(B)}   \no \\
& \quad \iff \sigma_{\rm ess} (B) = \emptyset, 
\end{align}
since $(B + (1 - c_b)I_{\cH})^{1/2} \colon \cH_b \to \cH$ is unitary by \eqref{B.59}. 
\end{proof}
%%%%%%%%%%%%%%%%%%%%

Throughout this paper we are employing the following notation: The
Banach spaces of bounded and compact linear operators on a Hilbert space $\cH$ are
denoted by $\cB(\cH)$ and $\cB_\infty(\cH)$, respectively. The analogous notation 
$\cB(\cX_1,\cX_2)$, $\cB_\infty (\cX_1,\cX_2)$, etc., will be used for bounded and compact 
operators between two Banach spaces $\cX_1$ and $\cX_2$. 
Moreover, $\cX_1\hookrightarrow \cX_2$ denotes the continuous embedding
of the Banach space $\cX_1$ into the Banach space $\cX_2$.

%%%%%%%%%%%%%%%%%%%%%%%%%%%%%%%%%%%%%%%%
%%%%%%%%%%%%%%%%%%%%%%%%%%%%%%%%%%%%%%%%
\section{Sobolev Spaces in Lipschitz Domains}
\label{s3}
%%%%%%%%%%%%%%%%%%%%%%%%%%%%%%%%%%%%%%%%
%%%%%%%%%%%%%%%%%%%%%%%%%%%%%%%%%%%%%%%%

The goal of this section is to introduce the relevant material pertaining
to Sobolev spaces $H^s(\Omega)$ and $H^r(\partial\Omega)$ corresponding to subdomains $\Om$ of $\bbR^n$, $n\in\bbN$, and discuss various trace results. 

We start by recalling some basic facts in connection with Sobolev spaces
corresponding to open subsets $\Om\subset\bbR^n$, $n\in\bbN$. For an arbitrary 
$m\in\bbN\cup\{0\}$, we follow the customary
way of defining $L^2$-Sobolev spaces of order $\pm m$ in $\Om$ as
\begin{align}\label{hGi-1}
H^m(\Om) &:= \big\{u\in L^2(\Om;d^nx)\,\big|\,\partial^\alpha u\in L^2(\Om;d^nx)
\mbox{ for } 0\leq|\alpha|\leq m\big\}, \\
H^{-m}(\Om) &:=\bigg\{u\in\cD^{\prime}(\Om)\,\bigg|\,u=\sum_{|\alpha|\leq m}
\partial^\alpha u_{\alpha}, \mbox{ with }u_\alpha\in L^2(\Om;d^nx), \, 
0\leq|\alpha|\leq m\bigg\},
\label{hGi-2}
\end{align}
equipped with natural norms (cf., e.g., \cite[Ch.\ 3]{AF03}, \cite[Ch.\ 1]{Ma85}). Here 
$\cD^\prime(\Om)$ denotes the usual set of distributions on $\Omega\subseteq \bbR^n$. Then one sets
\begin{equation}\label{hGi-3}
H^m_0(\Om):=\,\mbox{the closure of $C^\infty_0(\Om)$ in $H^m(\Om)$}, 
\quad m\in\bbN \cup \{0\}.
\end{equation}
As is well-known, all three spaces above are Banach, reflexive and, in addition,
\begin{equation}\label{hGi-4}
\big(H^m_0(\Om)\big)^*=H^{-m}(\Om).
\end{equation}
Again, see, for instance, \cite[Ch.\ 3]{AF03}, \cite[Sect.\ 1.1.14]{Ma85}. Throughout this paper, we agree to use the {\it adjoint} (rather than the dual) space $X^*$ of
a Banach space $X$.

One recalls that an open, nonempty, bounded set $\Omega\subset\bbR^n$ 
is called a {\it bounded Lipschitz domain} if the following property holds: 
There exists an open covering $\{{\mathcal O}_j\}_{1\leq j\leq N}$
of the boundary $\partial\Omega$ of $\Om$ such that for every
$j\in\{1,...,N\}$, ${\mathcal O}_j\cap\Omega$ coincides with the portion
of ${\mathcal O}_j$ lying in the over-graph of a Lipschitz function
$\varphi_j:\bbR^{n-1}\to\bbR$ (considered in a new system of coordinates
obtained from the original one via a rigid motion). The number
$\max\,\{\|\nabla\varphi_j\|_{L^\infty(\bbR^{n-1};d^{n-1}x')} \,|\,1\leq j\leq N\}$
is said to represent the {\it Lipschitz character} of $\Omega$.

The classical theorem of Rademacher of almost everywhere differentiability 
of Lipschitz functions ensures that, for any  Lipschitz domain $\Omega$, the
surface measure $d^{n-1} \omega$ is well-defined on  $\partial\Omega$ and
that there exists an outward  pointing normal vector $\nu$ at
almost every point of $\partial\Omega$.

In the remainder of this paper we shall make the following assumption:

%%%%%%%%%%%%%%%%%%%%%%%%%%%%%%%%%%%%%%%
\begin{hypothesis}\label{h2.1}
Let $n\in\bbN$, $n\geq 2$, and assume that $\Om\subset{\bbR}^n$ is
a bounded Lipschitz domain.
\end{hypothesis}
%%%%%%%%%%%%%%%%%%%%%%%%%%%%%%%%%%%%%%%

As regards $L^2$-based Sobolev spaces of fractional order $s\in\bbR$,
in a bounded {\it Lipschitz domain} $\Om\subset\bbR^n$ we set
\begin{align}\label{HH-h1}
H^{s}(\bbR^n) &:=\bigg\{U\in \cS^\prime(\bbR^n)\,\bigg|\,
\norm{U}_{H^{s}(\bbR^n)}^2 = \int_{\bbR^n}d^n\xi\,
\big|\hatt U(\xi)\big|^2\big(1+\abs{\xi}^{2s}\big)<\infty \bigg\},
\\
H^{s}(\Om) &:=\big\{u\in \cD^\prime(\Om)\,\big|\,u=U|_\Om\text{ for some }
U\in H^{s}(\bbR^n)\big\}. 
\label{HH-h2}
\end{align}
Here $\cS^\prime(\bbR^n)$ is the space of tempered distributions on $\bbR^n$,
and $\hatt U$ denotes the Fourier transform of $U\in\cS^\prime(\bbR^n)$.
These definitions are consistent with \eqref{hGi-1}--\eqref{hGi-2}.
Moreover, so is
\begin{equation}\label{incl-xxx}
H^{s}_0(\Omega):= \big\{u\in H^{s}(\bbR^n)\,\big|\, \supp(u)\subseteq\ol{\Omega}\big\},
\quad s\in\bbR,
\end{equation}
equipped with the natural norm induced by $H^{s}(\bbR^n)$, in relation to
\eqref{hGi-3}. One also has 
\begin{equation}\label{incl-Ya}
\big(H^{s}_0(\Omega)\big)^*=H^{-s}(\Omega),\quad s\in\bbR 
\end{equation}
(cf., e.g., \cite{JK95}). For a bounded Lipschitz domain $\Omega\subset\bbR^n$ 
it is known that
\begin{equation}\label{dual-xxx}
\bigl(H^{s}(\Omega)\bigr)^*=H^{-s}(\Omega), \quad - 1/2 <s< 1/2.
\end{equation}
See \cite{Tr02} for this and other related properties. 

To discuss Sobolev spaces on the boundary of a Lipschitz domain, consider
first the case when $\Omega\subset\bbR^n$ is the domain lying above the graph
of a Lipschitz function $\varphi\colon\bbR^{n-1}\to\bbR$. In this setting,
we define the Sobolev space $H^s(\partial\Omega)$ for $0\leq s\leq 1$,
as the space of functions $f\in L^2(\partial\Omega;d^{n-1}\omega)$ with the
property that $f(x',\varphi(x'))$, as a function of $x'\in\bbR^{n-1}$,
belongs to $H^s(\bbR^{n-1})$. This definition is easily adapted to the case
when $\Omega$ is a Lipschitz domain whose boundary is compact,
by using a smooth partition of unity. Finally, for $-1\leq s\leq 0$, we set
\begin{equation}\label{A.6}
H^s(\dOm) = \big(H^{-s}(\dOm)\big)^*, \quad -1 \le s \le 0.
\end{equation}
 From the above characterization of $H^s(\partial\Omega)$ it follows that
any property of Sobolev spaces (of order $s\in[-1,1]$) defined in Euclidean
domains, which are invariant under multiplication by smooth, compactly
supported functions as well as compositions by bi-Lipschitz diffeomorphisms,
readily extends to the setting of $H^s(\partial\Omega)$ (via localization and
pull-back). As a concrete example, for each Lipschitz domain $\Omega$ 
with compact boundary, one has  
\begin{equation} \label{EQ1}
H^s(\partial\Omega)\hookrightarrow L^2(\partial\Omega;d^{n-1} \omega)
\, \text{ compactly if }\,0<s\leq 1.  
\end{equation}
For additional background 
information in this context we refer, for instance, to \cite{Au04}, 
\cite{Au06}, \cite[Chs.\ V, VI]{EE89}, \cite[Ch.\ 1]{Gr85}, 
\cite[Ch.\ 3]{Mc00}, \cite[Sect.\ I.4.2]{Wl87}.

Assuming Hypothesis \ref{h2.1}, we introduce the boundary trace
operator $\ga_D^0$ (the Dirichlet trace) by
\begin{equation}
\ga_D^0\colon C(\ol{\Om})\to C(\dOm), \quad \ga_D^0 u = u|_\dOm.   \label{2.5}
\end{equation}
Then there exists a bounded linear operator $\gamma_D$
\begin{align}
\begin{split}
& \ga_D\colon H^{s}(\Om)\to H^{s-(1/2)}(\dOm) \hookrightarrow \LdOm,
\quad 1/2<s<3/2, \label{2.6}  \\
& \ga_D\colon H^{3/2}(\Om)\to H^{1-\varepsilon}(\dOm) \hookrightarrow \LdOm,
\quad \varepsilon \in (0,1) 
\end{split}
\end{align}
(cf., e.g., \cite[Theorem 3.38]{Mc00}), whose action is compatible with that of 
$\ga_D^0$. That is, the two Dirichlet trace  operators coincide on the intersection 
of their domains. Moreover, we recall that
\begin{equation}\label{2.6a}
\ga_D\colon H^{s}(\Om)\to H^{s-(1/2)}(\dOm) \, \text{ is onto for $1/2<s<3/2$}.
\end{equation}

Next, retaining Hypothesis \ref{h2.1}, we introduce the operator
$\ga_N$ (the strong Neumann trace) by
\begin{equation}\label{2.7}
\ga_N = \nu\cdot\ga_D\nabla \colon H^{s+1}(\Om)\to \LdOm, \quad 1/2<s<3/2, 
\end{equation}
where $\nu$ denotes the outward pointing normal unit vector to
$\partial\Om$. It follows from \eqref{2.6} that $\ga_N$ is also a
bounded operator. We seek to extend the action of the Neumann trace
operator \eqref{2.7} to other (related) settings. To set the stage,
assume Hypothesis \ref{h2.1} and observe that the inclusion
\begin{equation}\label{inc-1}
\iota:H^{s_0}(\Omega)\hookrightarrow \bigl(H^r(\Omega)\bigr)^*,\quad
s_0>-1/2, \; r>1/2,
\end{equation}
is well-defined and bounded. We then introduce the weak Neumann trace
operator
\begin{equation}\label{2.8}
\wti\ga_N\colon\big\{u\in H^{s+1/2}(\Om)\,\big|\, \Delta u\in H^{s_0}(\Om)\big\}
\to H^{s-1}(\dOm),\quad s\in(0,1),\; s_0>-1/2,
\end{equation}
as follows: Given $u\in H^{s+1/2}(\Om)$ with $\Delta u \in H^{s_0}(\Om)$
for some $s\in(0,1)$ and $s_0>-1/2$, we set (with $\iota$ as in
\eqref{inc-1} for $r:=3/2-s>1/2$)
\begin{equation}\label{2.9}
\langle\phi,\wti\ga_N u \rangle_{1-s}
={}_{H^{1/2-s}(\Om)}\langle\nabla\Phi,\nabla u\rangle_{(H^{1/2-s}(\Om))^*}
+ {}_{H^{3/2-s}(\Om)}\langle\Phi,\iota(\Delta u)\rangle_{(H^{3/2-s}(\Om))^*},
\end{equation}
for all $\phi\in H^{1-s}(\dOm)$ and $\Phi\in H^{3/2-s}(\Om)$ such that
$\ga_D\Phi=\phi$. We note that the first pairing on the right-hand side of  
\eqref{2.9} is meaningful since
\begin{equation}\label{2.9JJ}
\bigl(H^{1/2-s}(\Om)\bigr)^*=H^{s-1/2}(\Om),\quad s\in (0,1),
\end{equation}
and that the definition \eqref{2.9} is independent of the particular
extension $\Phi$ of $\phi$, and that $\wti\ga_N$ is a bounded extension
of the Neumann trace operator $\ga_N$ defined in \eqref{2.7}.

%%%%%%%%%%%%%%%%%%%%%%%%%%%%%%%%%%%%%%%%
%%%%%%%%%%%%%%%%%%%%%%%%%%%%%%%%%%%%%%%%
\section{Laplace Operators with Nonlocal Robin-Type \\ Boundary Conditions}
\label{s4}
%%%%%%%%%%%%%%%%%%%%%%%%%%%%%%%%%%%%%%%%
%%%%%%%%%%%%%%%%%%%%%%%%%%%%%%%%%%%%%%%%

In this section we primarily focus on various properties of general
Laplacians $-\Delta_{\Theta,\Om}$ in $L^2(\Om;d^n x)$ including Dirichlet,
$-\Delta_{D,\Om}$, and Neumann, $-\Delta_{N,\Om}$, Laplacians, nonlocal
Robin-type Laplacians, and Laplacians corresponding to classical Robin
boundary conditions associated with bounded Lipschitz domains 
$\Omega\subset \bbR^n$.

For simplicity of notation we will denote the identity operators in $\LOm$ and
$\LdOm$ by $I_{\Om}$ and $I_{\dOm}$, respectively. Also, in the sequel, the
sesquilinear form
\begin{equation}
\langle \dott, \dott \rangle_{s}={}_{H^{s}(\dOm)}\langle\dott,\dott
\rangle_{H^{-s}(\dOm)}\colon H^{s}(\dOm)\times H^{-s}(\dOm)
\to \bbC, \quad s\in [0,1],
\end{equation}
(antilinear in the first, linear in the second factor), will denote the duality
pairing between $H^s(\dOm)$ and
\begin{equation}\lb{2.3}
H^{-s}(\dOm)=\big(H^s(\dOm)\big)^*,\quad s\in [0,1],
\end{equation}
such that
\begin{align}\lb{2.4}
\begin{split} 
& \langle f,g\rangle_s=\int_{\dOm} d^{n-1}\omega(\xi)\,\ol{f(\xi)}g(\xi),  \\
& f\in H^s(\dOm),\,g\in L^2(\dOm;d^{n-1}\omega)\hookrightarrow 
H^{-s}(\dOm), \; s\in [0,1],
\end{split} 
\end{align}
where, as before, $d^{n-1}\omega$ stands for the surface measure on $\dOm$.

We also recall the notational conventions summarized at the end of Section \ref{s2}. 

%%%%%%%%%%%%%%%%%%%%%%%%%%%%%
\begin{hypothesis} \lb{h2.2}
Assume Hypothesis \ref{h2.1}, suppose that $\delta>0$ is a given number,
and assume that $\Theta\in\cB\big(H^{1/2}(\dOm),H^{-1/2}(\dOm)\big)$
is a self-adjoint operator which can be written as
\begin{equation}\label{Filo-1}
\Theta=\Theta_1+\Theta_2+\Theta_3,
\end{equation}
where $\Theta_j$, $j=1,2,3$, have the following properties: There exists a closed sesquilinear form $a_{\Theta_0}$ in $\LdOm$,
with domain $H^{1/2}(\dOm)\times H^{1/2}(\dOm)$, bounded from below
by $c_{\Theta_0}\in\bbR$ $($hence, $a_{\Theta_0}$
is symmetric$)$ such that if $\Theta_0\ge c_{\Theta_0}I_{\dOm}$ denotes
the self-adjoint operator in $\LdOm$ uniquely associated with $a_{\Theta_0}$
$($cf. \eqref{B.25}$)$, then $\Theta_1=\wti\Theta_0$, the extension
of $\Theta_0$ to an operator in $\cB\big(H^{1/2}(\dOm),H^{-1/2}(\dOm)\big)$
$($as discussed in \eqref{B.24a} and \eqref{B.28a}$)$. In addition,
\begin{equation}\label{Filo-2}
\Theta_2\in\cB_{\infty}\big(H^{1/2}(\dOm),H^{-1/2}(\dOm)\big),
\end{equation}
whereas $\Theta_3\in\cB\big(H^{1/2}(\dOm),H^{-1/2}(\dOm)\big)$ satisfies
\begin{equation}\label{Filo-3}
\|\Theta_3\|_{\cB(H^{1/2}(\dOm),H^{-1/2}(\dOm))}<\delta.
\end{equation}
\end{hypothesis}
%%%%%%%%%%%%%%%%%%%%%%%%%%%%%%

We recall the following useful result.

%%%%%%%%%%%%%%%%%%%%
\begin{lemma}\lb{l2.2a}
Assume Hypothesis \ref{h2.1}. Then for every $\varepsilon >0$
there exists a $\beta(\varepsilon)>0$
$($with $\beta(\varepsilon)\underset{\varepsilon\downarrow 0}{=}
\Oh(1/\varepsilon)$$)$ such that
\begin{equation}\lb{2.38c}
\|\gamma_D u\|_{\LdOm}^2 \le
\varepsilon \|\nabla u\|_{\LOm^n}^2 + \beta(\varepsilon) \|u\|_{\LOm}^2, 
\quad u\in H^1(\Om).
\end{equation}
\end{lemma}
%%%%%%%%%%%%%%%%%%%%

A proof from which it is possible to read off how the constant
$\beta(\varepsilon)$ depends on the Lipschitz character of $\Om$ appears
in \cite{GM08a}. Below we discuss a general abstract scheme which yields
results of this type, albeit with a less descriptive constant
$\beta(\varepsilon)$. The lemma below is inspired by \cite[Lemma\ 2.3]{AM07}: 

%%%%%%%%%%%%%%%%%%%%
\begin{lemma}\lb{l.WW}
Let $\cV$ be a reflexive Banach space, $\cW$ a Banach space, 
assume that $K\in\cB_{\infty}(\cV,\cV^*)$, and that $T\in\cB(\cV,\cW)$ is
one-to-one. Then for every $\varepsilon>0$ there exists $C_{\varepsilon}>0$
such that
\begin{equation}\lb{K-u1}
\big|{}_{\cV}\langle u,Ku\rangle_{\cV^*}\big|\leq\varepsilon\|u\|^2_{\cV}
+C_{\varepsilon}\|Tu\|^2_{\cW},\quad u\in\cV.
\end{equation}
\end{lemma}
%%%%%%%%%%%%%%%%%%%%%
\begin{proof}
Seeking a contradiction, assume that there exist $\varepsilon>0$
along with a family of vectors $u_j\in\cV$, $\|u_j\|_{\cV}=1$, $j\in\bbN$,
for which
\begin{equation}\lb{K-u2}
\big|{}_{\cV}\langle u_j,Ku_j \rangle_{\cV^*}\big|\geq\varepsilon
+j\|Tu_j\|^2_{\cW},\quad j\in\bbN.
\end{equation}
Furthermore, since $\cV$ is reflexive, there is no loss of generality in
assuming that there exists $u\in\cV$ such that $u_j\to u$ as $j\to\infty$,
weakly in $\cV$ (cf., e.g., \cite[Theorem\ 1.13.5]{Me98}). In addition, since $T$ (and hence $T^*$) is bounded, one concludes that $Tu_j\to Tu$ as $j\to\infty$ weakly in $\cW$, as is clear from 
\begin{equation}
{}_{\cW}\langle T u_j, \Lambda\rangle_{\cW^*} 
= {}_{\cV}\langle u_j, T^* \Lambda\rangle_{\cV^*} \underset{j\to \infty} \longrightarrow 
{}_{\cV}\langle u, T^* \Lambda\rangle_{\cV^*} 
= {}_{\cW}\langle T u, \Lambda\rangle_{\cW^*}, \quad \Lambda \in \cW^*.
\end{equation}  
Moreover, since $K$ is compact, we may choose a subsequence of 
$\{u_j\}_{j\in\bbN}$ (still denoted by $\{u_j\}_{j\in\bbN}$) such that 
$Ku_j\to Ku$ as $j\to\infty$, strongly in $\cV^*$. This, in turn, yields  that
\begin{equation}\lb{K-u3}
{}_{\cV}\langle u_j,Ku_j\rangle_{\cV^*}\to
{}_{\cV}\langle u,Ku\rangle_{\cV^*} \,\mbox{ as } \, j\to\infty.
\end{equation}
Together with
\begin{equation}\lb{K-u4}
\|Tu_j\|^2_{\cW}\leq j^{-1}
\big|{}_{\cV}\langle u_j,Ku_j\rangle_{\cV^*}\big|,
\quad j\in\bbN,
\end{equation}
this also shows that $Tu_j\to 0$ as $j\to\infty$, in $\cW$.
Hence, $Tu=0$ in $\cW$ which forces $u=0$, since $T$ is one-to-one.
Given these facts, we note that, on the one hand,
we have ${}_{\cV}\langle u_j,Ku_j\rangle_{\cV^*}\to 0$ as $j\to\infty$
by \eqref{K-u3}, while on the other hand
$ \Bigl|{}_{\cV}\langle u_j,Ku_j\rangle_{\cV^*}\Bigr|\geq\varepsilon$
for every $j\in\bbN$ by \eqref{K-u2}. This contradiction concludes the proof. 
\end{proof}
%%%%%%%%%%%%%%%%%%%%%%%%%%

Parenthetically, we note that Lemma \ref{l2.2a} (with a less
precise description of the constant $\beta(\varepsilon)$) follows
from Lemma \ref{l.WW} by taking
\begin{equation}\lb{K-u5}
\cV:=H^1(\Om),\quad \cW:=L^2(\Om,d^nx),
\end{equation}
and, with $\gamma_D\in\cB_{\infty}\big(H^1(\Om),L^2(\dOm;d^{n-1}\omega)\big)$
denoting the Dirichlet trace,
\begin{equation}\lb{K-u6}
K:=\gamma_D^*\gamma_D\in\cB_{\infty}\big(H^1(\Om),\bigl(H^1(\Om)\bigr)^*\big),
\quad T:=\iota:H^1(\Om)\hookrightarrow L^2(\Om,d^nx),
\end{equation}
the inclusion operator.

%%%%%%%%%%%%%%%%%%%%%%%%%%%
\begin{lemma}\lb{l.WW2}
Assume Hypothesis \ref{h2.2}, where the number $\delta>0$ is taken
to be sufficiently small relative to the Lipschitz character of $\Om$.
Consider the sesquilinear form $a_\Theta(\dott,\dott)$ defined
on $H^1(\Om)\times H^1(\Om)$ by
\begin{equation}\lb{2.22}
a_\Theta(u,v):=\int_{\Om} d^nx\,\ol{(\nabla u)(x)}\cdot(\nabla v)(x)
+\big\langle\gamma_D u,\Theta\gamma_D v \big\rangle_{1/2},
\quad u,v\in H^1(\Om).
\end{equation}
Then there exists $\kappa>0$ with the property that the form
\begin{equation}\label{FI-5}
a_{\Theta,\kappa}(u,v):=a_\Theta(u,v)+\kappa\,(u,v)_{L^2(\Om;d^nx)},
\quad u,v\in H^1(\Om),
\end{equation}
is $H^1(\Om)$-coercive.

As a consequence, the form \eqref{2.22} is symmetric, $H^1(\Om)$-bounded,
bounded from below, and closed in $L^2(\Om; d^n x)$.
\end{lemma}
%%%%%%%%%%%%%%%%%%%%%%%%%%%
\begin{proof}
We shall show that $\kappa>0$ can be chosen large enough so that
\begin{align}\lb{2.22Y.1}
\begin{split} 
&\frac16\|u\|^2_{H^1(\Om)}\leq\frac{1}{3}\int_{\Om}d^nx\,|(\nabla u)(x)|^2
+\frac{\kappa}{3}\int_{\Om}d^nx\,|u(x)|^2+
\big\langle\gamma_D u,\Theta_j\gamma_D v \big\rangle_{1/2},  \\
& \hspace*{7.7cm}  u\in H^1(\Om), \; j=1,2,3, 
\end{split}
\end{align}
where $\Theta_j$, $j=1,2,3$, are as introduced in Hypothesis \ref{h2.2}.
Summing up these three inequalities then proves that the form \eqref{FI-5} is
indeed $H^1(\Om)$-coercive. To this end, we assume first $j=1$ and recall
that there exists $c_{\Theta_0}\in\bbR$ such that
\begin{equation}\label{PdPd-1}
\big\langle\ga_D u,\Theta_1\,\ga_D u\big\rangle_{1/2}
\geq c_{\Theta_0}\|\ga_D 
u\|_{L^2(\dOm;d^{n-1}\omega)}^2,\quad u\in H^1(\Om).
\end{equation}
Thus, in this case, it suffices to show that
\begin{align}\lb{PdPd-2}
& \max\,\{-c_{\Theta_0}\,,\,0\}\,\|\ga_D u\|_{L^2(\dOm;d^{n-1}\omega)}^2
+\frac{1}{6}\|u\|^2_{H^1(\Om)}
\no \\
&\quad 
\leq\frac{1}{3}\int_{\Om}d^nx\,|(\nabla u)(x)|^2
+\frac{\kappa}{3}\int_{\Om}d^nx\,|u(x)|^2,
\quad u\in H^1(\Om),
\end{align}
or, equivalently, that  
\begin{align}\lb{PdPd-3}
& \max\,\{-c_{\Theta_0}\,,\,0\}\,\|\ga_D u\|_{L^2(\dOm;d^{n-1}\omega)}^2
\no \\
& \quad
\leq\frac{1}{6}\int_{\Om}d^nx\,|(\nabla u)(x)|^2
+\frac{2\kappa-1}{6}\int_{\Om}d^nx\,|u(x)|^2,\quad u\in H^1(\Om),
\end{align}
with the usual convention,
\begin{equation}
\|u\|^2_{H^1(\Om)} = \|\nabla u\|^2_{_{L^2(\Om;d^n  x)^n}} 
+ \|u\|^2_{_{L^2(\Om;d^n x)}}, \quad u\in H^1(\Om).
\end{equation}
The fact that there exists $\kappa>0$ for which \eqref{PdPd-3} holds
follows directly from Lemma \ref{l2.2a}.

Next, we observe that in the case where $j=2,3$, estimate \eqref{2.22Y.1} is
implied by
\begin{equation}\lb{2.22Y.2}
\big|\big\langle\gamma_D u,\Theta_j\gamma_D u\big\rangle_{1/2}\big|
\leq\frac{1}{6}\int_{\Om}d^nx\,|(\nabla u)(x)|^2
+\frac{2\kappa-1}{6}\int_{\Om}d^nx\,|u(x)|^2,\quad u\in H^1(\Om),
\end{equation}
or, equivalently, by 
\begin{equation}\lb{2.22Y.3}
\big|\big\langle\gamma_D u,\Theta_j\gamma_D u\big\rangle_{1/2}\big|
\leq\frac{1}{6}\|u\|^2_{H^1(\Om)}
+\frac{\kappa-1}{3}\|u\|^2_{L^2(\Om;d^nx)},\quad u\in H^1(\Om).
\end{equation}
When $j=2$, in which case 
$\Theta_2\in\cB_{\infty}\big(H^1(\Om),\bigl(H^1(\Om)\bigr)^*\big)$,
we invoke Lemma \ref{l.WW} with $\cV$, $\cW$ as in \eqref{K-u5} 
and, with $\gamma_D\in\cB\big(H^1(\Om),H^{1/2}(\dOm)\big)$
denoting the Dirichlet trace,
\begin{equation}\lb{PdPd-4}
K:=\gamma_D^*\Theta_2\gamma_D
\in\cB_{\infty}\big(H^1(\Om),\bigl(H^1(\Om)\bigr)^*\big),
\quad T:=\iota:H^1(\Om)\hookrightarrow L^2(\Om,d^nx),
\end{equation}
the inclusion operator. Then, with $\varepsilon=1/6$ and
$\kappa:=3C_{1/6}+1$, estimate \eqref{K-u1} yields \eqref{2.22Y.3} for $j=2$.

Finally, consider \eqref{2.22Y.3} in the case where $j=3$ and note that 
by hypothesis,
\begin{align}\lb{2.22Y.4}
\big|\big\langle\gamma_D u,\Theta_3\gamma_D u\big\rangle_{1/2}\big|
&\leq 
\|\Theta_3\|_{\cB(H^{1/2}(\dOm),H^{-1/2}(\dOm))}\|\gamma_D 
u\|^2_{H^{1/2}(\dOm)}
\no \\
&\leq 
\delta \|\gamma_D\|^2_{\cB(H^1(\Om),H^{1/2}(\dOm))}\|u\|^2_{H^1(\Om)}, 
\quad u\in H^1(\Om).
\end{align}
Thus \eqref{2.22Y.3} also holds for $j=3$ if
\begin{equation}\lb{2.22Y.5}
0<\delta \leq \f{1}{6} \|\gamma_D\|^{-2}_{\cB(H^1(\Om),H^{1/2}(\dOm))} \, 
\text{ and } \, \kappa>1.  
\end{equation}
This completes the justification of \eqref{2.22Y.1}, and hence 
finishes the proof.  
\end{proof}
%%%%%%%%%%%%%%%%%%%%%%%

Next, we turn to a discussion of nonlocal Robin Laplacians in bounded 
Lipschitz subdomains of $\bbR^n$. Concretely, we describe a family of
self-adjoint Laplace operators $-\Delta_{\Theta,\Om}$ in $L^2(\Om;d^nx)$
indexed by the boundary operator $\Theta$. We will refer to
$-\Delta_{\Theta,\Om}$ as the nonlocal Robin Laplacian.

%%%%%%%%%%%%%%%%%%%%
\begin{theorem}\lb{t2.3}
Assume Hypothesis \ref{h2.2}, where the number $\delta>0$ is taken
to be sufficiently small relative to the Lipschitz character of $\Om$.
Then the nonlocal Robin Laplacian, $-\Delta_{\Theta,\Om}$, defined by
\begin{align}\lb{2.20}
& -\Delta_{\Theta,\Om}=-\Delta,    \\
& \, \dom(-\Delta_{\Theta,\Om})=
\big\{u\in H^1(\Om) \,\big|\, \Delta u\in L^2(\Om;d^nx),\,
\big(\wti\gamma_N+\Theta \gamma_D\big)u=0\text{ in $H^{-1/2}(\dOm)$}\big\} \no
\end{align}
is self-adjoint and bounded from below in $L^2(\Om;d^nx)$. Moreover,
\begin{equation}\lb{2.21}
\dom\big(|-\Delta_{\Theta,\Om}|^{1/2}\big)=H^1(\Om),
\end{equation}
and $-\Delta_{\Theta,\Om}$, has purely discrete spectrum bounded from below,
in particular,
\begin{equation}\lb{2.21a}
\sigma_{\rm ess}(-\Delta_{\Theta,\Om})=\emptyset.
\end{equation}
Finally, $-\Delta_{\Theta,\Om}$ is the operator uniquely associated with the sesquilinear 
form $a_{\Theta}$ in Lemma \ref{l.WW2}. 
\end{theorem}
%%%%%%%%%%%%%%%%%%%%
\begin{proof}
Denote by $a_{-\Delta_{\Theta,\Om}}(\dott,\dott)$ the sesquilinear form
introduced in \eqref{2.22}. From Lemma \ref{l.WW2}, we know that
$a_{-\Delta_{\Theta,\Om}}$ is symmetric, $H^1(\Om)$-bounded,
bounded from below, as well as densely defined and closed in $\LOm\times\LOm$.
Thus, if as in \eqref{B.30}, we now introduce the operator
$-\Delta_{\Theta,\Om}$ in $L^2(\Om;d^n x)$ by
\begin{align}\lb{2.31}
& \dom(-\Delta_{\Theta,\Om})=\bigg\{v\in H^1(\Om) \,\bigg|\, 
\text{there exists a $w_v \in L^2(\Om;d^n x)$ such that}
\nonumber\\[4pt]
& \quad \;\; \int_{\Om}d^nx\,\ol{\nabla w} \, \nabla v
+ \big\langle\gamma_D w,\Theta\gamma_D v\big\rangle_{1/2}
=\int_{\Om}d^n x\,\ol{w}w_v\text{ for all $w\in H^1(\Om)$}\bigg\},
\nonumber\\[4pt]
& -\Delta_{\Theta,\Om}u=w_u,\quad u\in\dom(-\Delta_{\Theta,\Om}),
\end{align}
it follows from \eqref{B.19}--\eqref{B.40} (cf., in particular \eqref{B.25})
that $-\Delta_{\Theta,\Om}$ is self-adjoint and bounded from below in
$L^2(\Om;d^n x)$ and that \eqref{2.21} holds. Next we recall that
\begin{equation}\label{H-zer}
H_0^1(\Om)=\big\{u\in H^1(\Om)\,\big|\,\gamma_Du=0\mbox{ on }\partial\Omega\big\}.
\end{equation}
Taking $v\in C_0^\infty(\Omega) \hookrightarrow H^1_0(\Om)
\hookrightarrow H^1(\Om)$, one concludes
\begin{equation}\lb{2.32}
\int_{\Om}d^nx\,{\ol v}w_u=-\int_{\Om}d^nx\,{\ol v}\,\Delta u\,\
\text{ for all $v\in C_0^\infty(\Om)$, and hence }\,w_u
=-\Delta u\,\text{ in }\,\cD^\prime(\Om),
\end{equation}
with $\cD^\prime(\Om)=C_0^\infty(\Omega)^\prime$ the space of distributions
on $\Om$. Going further, suppose that $u\in\dom(-\Delta_{\Theta,\Om})$ and
$v\in H^1(\Om)$. We recall that $\gamma_D\colon H^1(\Om)\to H^{1/2}(\dOm)$
and compute
\begin{align}
\int_{\Om}d^nx\,\ol{\nabla v}\,\nabla u
& =-\int_{\Om}d^n x\,{\ol v}\,\Delta u
+\langle\gamma_D v,\wti\gamma_N u\rangle_{1/2}
\nonumber\\[4pt]
& =\int_{\Om}d^n x\,{\ol v} w_u
+\big\langle\gamma_D v,\big(\wti\gamma_N+\Theta\gamma_D\big)u\big\rangle_{1/2}
-\big\langle\gamma_D v,\Theta\gamma_D u\big\rangle_{1/2}
\nonumber\\[4pt]
& =\int_{\Om}d^n x\,\ol{\nabla v}\,\nabla u
+\big\langle\gamma_D v,\big(\wti\gamma_N+\Theta\gamma_D\big)u\big\rangle_{1/2},
\lb{2.33}
\end{align}
where we used the second line in \eqref{2.31}. Hence,
\begin{equation}
\big\langle \gamma_D v, \big(\wti\gamma_N+\Theta\gamma_D\big) u
\big\rangle_{1/2} =0.
\lb{2.34}
\end{equation}
Since $v\in H^1(\Om)$ is arbitrary, and the map
$\gamma_D\colon H^1(\Om) \to H^{1/2}(\dOm)$ is actually onto,
one concludes that
\begin{equation}
\big(\wti\gamma_N+\Theta\gamma_D\big)u=0\,\text{ in }\,H^{-1/2}(\dOm).
\lb{2.35}
\end{equation}
Thus,
\begin{equation}\lb{2.36}
\dom(-\Delta_{\Theta,\Om}) \subseteq \big\{v\in H^1(\Om) \,\big|\,
\Delta v\in L^2(\Om; d^nx),\,\big(\wti\gamma_N+\Theta\gamma_D\big)v=0
\text{ in } H^{-1/2}(\dOm)\big\}.
\end{equation}
Next, assume that $u\in \big\{v\in H^1(\Om) \,\big|\, \Delta v\in L^2(\Om;d^n x),\,
\big(\wti\gamma_N+\Theta\gamma_D\big)v=0\big\}$, $w\in H^1(\Om)$,
and let $w_u=-\Delta u\in L^2(\Om; d^n x)$. Then,
\begin{align}
\int_{\Om} d^n x \, {\ol w} w_u
&= - \int_{\Om} d^n x \, {\ol w} \, {\rm div} (\nabla u)   \no \\
&= \int_{\Om} d^n x \, \ol{\nabla w} \, \nabla u
- \langle \gamma_D w, \wti \gamma_N u \rangle_{1/2} \no \\
&= \int_{\Om} d^n x \, \ol{\nabla w} \, \nabla u
+ \big\langle \gamma_D w,\Theta \gamma_D u \big\rangle_{1/2}.
\lb{2.37}
\end{align}
Thus, applying \eqref{2.31}, one concludes that
$u \in \dom(-\Delta_{\Theta,\Om})$ and hence
\begin{equation}\lb{2.38}
\dom(-\Delta_{\Theta,\Om}) \supseteq \big\{v\in H^1(\Om)\,\big|\,
\Delta v \in L^2(\Om; d^n x), \,\big(\wti\gamma_N+\wti\Theta\gamma_D\big) v=0
  \text{ in } H^{-1/2}(\dOm)\big\}.
\end{equation}
Finally, the last claim in the statement of Theorem \ref{t2.3} follows
from the fact that $H^1(\Om)$ embeds compactly into $\LOm$ (cf., e.g.,
\cite[Theorem V.4.17]{EE89}); see Lemma \ref{lB.1}.
\end{proof}
%%%%%%%%%%%%%%%%%%%%

In the special case $\Theta=0$, that is, in the case of the Neumann Laplacian,
we will also use the notation
\begin{equation}\lb{2.38b}
-\Delta_{N,\Om}:=-\Delta_{0,\Om}.
\end{equation}
The case of the Dirichlet Laplacian $-\Delta_{D,\Om}$ associated with
$\Om$ formally corresponds to $\Theta = \infty$ and so we isolate it
in the next result (cf.\ also \cite{GLMZ05}, \cite{GMZ07}):

%%%%%%%%%%%%%%%%%%%%
\begin{theorem}  \lb{t2.5}
Assume Hypothesis \ref{h2.1}. Then the Dirichlet Laplacian,
$-\Delta_{D,\Om}$, defined by
\begin{align}
& -\Delta_{D,\Om} = -\Delta,   \no \\
& \; \dom(-\Delta_{D,\Om}) =
\big\{u\in H^1(\Om)\,\big|\, \Delta u \in L^2(\Om;d^n x), \,
\gamma_D u =0 \text{ in $H^{1/2}(\dOm)$}\big\}    \lb{2.39} \\
& \hspace*{2.13cm} = \big\{u\in H_0^1(\Om)\,\big|\, \Delta u \in L^2(\Om;d^n x)\big\},  \no 
\end{align}
is self-adjoint and strictly positive in $L^2(\Om;d^nx)$. Moreover,
\begin{equation}
\dom\big((-\Delta_{D,\Om})^{1/2}\big) = H^1_0(\Om).   \lb{2.40}
\end{equation}
\end{theorem}
%%%%%%%%%%%%%%%%%%%%

Since $\Om$ is open and bounded, it is well-known that $-\Delta_{D,\Om}$ has
purely discrete spectrum contained in $(0,\infty)$, in particular,
\begin{equation}
\sigma_{\rm ess}(-\Delta_{D,\Om})=\emptyset. 
\end{equation}
This follows from \eqref{2.40} since $H^1_0(\Om)$ embeds compactly into $\LOm$; the latter fact holds for arbitrary open, bounded sets $\Om\subset\bbR^n$ (see, e.g., 
\cite[Theorem V.4.18]{EE89}).

%%%%%%%%%%%%%%%%%%%%%%%%%%%%%%%%%%%%%%%%
%%%%%%%%%%%%%%%%%%%%%%%%%%%%%%%%%%%%%%%%
\section{Eigenvalue Inequalities}
\label{s5}
%%%%%%%%%%%%%%%%%%%%%%%%%%%%%%%%%%%%%%%%
%%%%%%%%%%%%%%%%%%%%%%%%%%%%%%%%%%%%%%%%

Assume Hypothesis \ref{h2.2} and denote by
\begin{equation}\label{mA.1}
\lambda_{\Theta,\Om,1}\leq\lambda_{\Theta,\Om,2}\leq\cdots\leq
\lambda_{\Theta,\Om,j}\leq\lambda_{\Theta,\Om,j+1}\leq\cdots
\end{equation}
the eigenvalues for the Robin Laplacian $-\Delta_{\Theta,\Om}$  in $L^2(\Om; d^n x)$, listed according to their multiplicity. Similarly, we let
\begin{equation}\label{mA.2}
0<\lambda_{D,\Om,1} < \lambda_{D,\Om,2}\leq\cdots\leq\lambda_{D,\Om,j}
\leq\lambda_{D,\Om,j+1}\leq\cdots
\end{equation}
be the eigenvalues for the Dirichlet Laplacian $-\Delta_{D,\Om}$  in $L^2(\Om; d^n x)$, again enumerated according to their multiplicity.

%%%%%%%%%%%%%%%%%%%%
\begin{theorem}\lb{t5.1A}
Assume Hypothesis \ref{h2.2}, where the number $\delta>0$ is taken
to be sufficiently small relative to the Lipschitz character of $\Om$
and, in addition, suppose that
\begin{equation}\label{FGM-1}
\big\langle\gamma_D(e^{ix\cdot\eta}),\Theta\gamma_D(e^{ix\cdot\eta})\big\rangle_{1/2}
\leq 0 \, \text{ for all } \, \eta\in\bbR^n.
\end{equation}
Then
\begin{equation}\label{FGM-2}
\lambda_{\Theta,\Om,j+1}<\lambda_{D,\Om,j},\quad j\in\bbN.
\end{equation}
\end{theorem}
%%%%%%%%%%%%%%%%%%%%
\begin{proof}
One can follow Filonov \cite{Fi04} closely. The main reason we present
Filonov's elegant argument is to ensure that this continues to hold
in the case when a nonlocal Robin boundary condition is considered (in lieu
of the Neumann boundary condition).
Recalling the eigenvalue counting functions for the Dirichlet and Robin Laplacians,
one sets for each $\lambda\in\bbR$,  
\begin{equation}\label{FGM-3}
N_D(\lambda):=\#\,\{\sigma(-\Delta_{D,\Om})\cap (-\infty,\lambda]\},
\quad
N_\Theta(\lambda):=\#\,\{\sigma(-\Delta_{\Theta,\Om})\cap (-\infty,\lambda]\}.
\end{equation}
Then Lemmas \ref{lB.1} and \ref{l.WW2} ensure that 
for each $\lambda\in\bbR$ one has
\begin{align}\label{FGN-4}
& N_D(\lambda)=\max \bigg\{\dim (L)\in\bbN_0 \,\bigg|\,L\text{ a subspace of }
H^1_0(\Om) \text{ such that }  \no \\
& \hspace*{3.85cm} 
\int_{\Om}d^nx\,|(\nabla u)(x)|^2
\leq\lambda\|u\|^2_{L^2(\Om;d^nx)}\text{ for all }u\in L\bigg\},
\end{align}
and
\begin{align}\label{FGN-5}
& N_\Theta(\lambda)
=\max \bigg\{\dim (L)\in\bbN_0 \,\bigg|\,L\text{ a subspace of }H^1(\Om)
\text{ with the property that}   \no \\
& \hspace*{1cm} 
\int_{\Om}d^nx\,|(\nabla u)(x)|^2
+\langle\ga_D u,\Theta\ga_D u\rangle_{1/2}
\leq\lambda\|u\|^2_{L^2(\Om;d^nx)}\text{ for all }u\in L\bigg\}.  
\end{align}
Next, observe that for any $\lambda\in\bbC$,
\begin{equation}\label{FGN-6}
H^1_0(\Om)\cap \ker(-\Delta_{\Theta,\Om}-\lambda\,I_{\Om})=\{0\}.
\end{equation}
Indeed, if
$u\in H^1_0(\Om)\cap \ker(-\Delta_{\Theta,\Om}-\lambda\,I_{\Om})$, 
then $u\in H^1(\Om)$ satisfies $(-\Delta-\lambda)u=0$ in $\Om$
and $\ga_D u = \wti\ga_N u=0$. It follows that the extension by
zero of $u$ to the entire $\bbR^n$ belongs to $H^1(\bbR^n)$, is compactly
supported, and is annihilated by $-\Delta-\lambda$. Hence, this function
vanishes identically, by unique continuation (see, e.g., \cite[p.\ 239--244]{RS78}).

To continue, we fix $\lambda>0$ and pick a subspace $U_{\lambda}$ of 
$H^1_0(\Om)$ such that $\dim(U_{\lambda})=N_D(\lambda)$ and
\begin{equation}\label{FGN-7}
\int_{\Om}d^nx\,|(\nabla u)(x)|^2\leq\lambda\,\int_{\Om}d^nx\,|u(x)|^2, 
\quad u\in U_{\lambda}.
\end{equation}
Then the sum $U_{\lambda} \, \dot{+} \, \ker(-\Delta_{\Theta,\Om}-\lambda\,I_{\Om})$
is direct, by \eqref{FGN-6}. Since the functions
$\big\{e^{ix\cdot\eta} \,\big|\, \eta\in\bbR^n,\,|\eta|=\sqrt{\lambda}\big\}$ are
linearly independent, it follows that there exists a vector $\eta_0\in\bbR^n$
with $|\eta_0|=\sqrt{\lambda}$ and such that $e^{ix\cdot\eta_0}$ does not
belong to the finite-dimensional space
$U_{\lambda} \, \dot{+} \, \ker(-\Delta_{\Theta,\Om}-\lambda\,I_{\Om})$. Assuming that
this is the case, introduce
\begin{equation}\label{FGN-8}
W_{\lambda}:= U_{\lambda} \, \dot{+} \, \ker(-\Delta_{\Theta,\Om}-\lambda\,I_{\Om})
\, \dot{+} \, \big\{ce^{ix\cdot\eta_0} \,\big|\, c\in\bbC\big\},
\end{equation}
so that $W_{\lambda}$ is a finite-dimensional subspace of $H^1(\Om)$.
Let $w=u+v+ce^{ix\cdot\eta_0}$ be an arbitrary vector in $W_{\lambda}$,
where $u\in U_{\lambda}$, $v\in \ker(-\Delta_{\Theta,\Om}-\lambda\,I_{\Om})$,
and $c\in\bbC$. We then write
\begin{align}\label{FGN-9}
& \int_{\Om}d^nx|(\nabla w)(x)|^2+\langle\ga_D w,\Theta\ga_D w\rangle_{1/2}
\no \\
& \quad =\int_{\Om}d^nx|\nabla(u+v+ce^{ix\cdot\eta_0})|^2
+\langle\ga_D(v+ce^{ix\cdot\eta_0}),\Theta\ga_D (v+ce^{ix\cdot\eta_0})\rangle_{1/2}
\no \\
& \quad =\int_{\Om}d^nx \, \big(|\nabla u|^2+|\nabla v|^2+|c\eta_0|^2\big)  \no \\
& \qquad +2 \Re \, \bigg(\int_{\Om}d^nx\, 
\big[\ol{\nabla v}\cdot\nabla(u+ce^{ix\cdot\eta_0})
 +\ol{\nabla(ce^{ix\cdot\eta_0})}\cdot\nabla u\big]\bigg) \no \\
& \qquad 
+\langle\ga_D(v+ce^{ix\cdot\eta_0}),\Theta\ga_D (v+ce^{ix\cdot\eta_0})\rangle_{1/2}
\no \\
& \quad =:I_1+I_2+I_3.
\end{align}
An integration by parts shows that
\begin{align}\label{FGN-9X}
\int_{\Om}d^nx\,|\nabla v|^2 &= -\int_{\Om}d^nx\, {\ol v} \Delta v 
+\langle\ga_D v,\wti\gamma_N v\rangle_{1/2}
\no \\
&= \lambda\,\int_{\Om}d^nx\,|v|^2-\langle\ga_D v,\Theta\ga_D v\rangle_{1/2}
\end{align}
where the last equality holds thanks to $-\Delta v=\lambda\,v$ and
$\wti\ga_N v=-\Theta\ga_D v$.
We now make use of this, \eqref{FGN-7}, the fact that $|\eta_0|^2=\lambda$,
in order to estimate
\begin{equation}\label{FGN-9Y}
I_1\leq \lambda\,\int_{\Om}d^nx\, \big[|u|^2+|v|^2+|c|^2\big]
-\langle\ga_D v,\Theta\ga_D v\rangle_{1/2}.
\end{equation}
Similarly,
\begin{align}\label{FGN-9Z}
I_2 &= -2 \Re \, \bigg(\int_{\Om}d^nx\,[\ol{\Delta v}(u+ce^{ix\cdot\eta_0})
+\ol{\Delta(ce^{ix\cdot\eta_0})}u]\bigg) 
+2 \Re \, \big(\langle\ga_D (ce^{ix\cdot\eta_0}),\wti\ga_N v\rangle_{1/2}\big)
\no \\
&= 2\lambda \Re \, \bigg(\int_{\Om}d^nx\,[\ol{v}(u+ce^{ix\cdot\eta_0})
+\ol{ce^{ix\cdot\eta_0}}u]\bigg)
-2 \Re \, \big(\langle\ga_D (ce^{ix\cdot\eta_0}),\Theta\ga_D v\rangle_{1/2}\big).
\end{align}
Thus, altogether,
\begin{align}\label{FGN-9W}
\begin{split}
& \int_{\Om}d^nx|(\nabla w)(x)|^2+\langle\ga_D w,\Theta\ga_D w\rangle_{1/2}  \\
& \quad \leq \lambda\,\int_{\Om}d^nx|w(x)|^2
+|c|^2\langle\ga_D(e^{ix\cdot\eta_0}),\Theta\ga_D (e^{ix\cdot\eta_0})\rangle_{1/2}.
\end{split} 
\end{align}
Upon recalling \eqref{FGM-1}, this yields
\begin{equation}\label{FGN-10}
\int_{\Om}d^nx|(\nabla w)(x)|^2+\langle\ga_D w,\Theta\ga_D w\rangle_{1/2}
\leq \lambda\,\int_{\Om}d^nx|w(x)|^2,\quad w\in W_{\lambda}.
\end{equation}
Consequently,
\begin{align}\label{FGN-11}
N_{\Theta}(\lambda)\geq \dim(W_{\lambda}) &= \dim(U_{\lambda}) 
+ \dim(\ker (-\Delta_{\Theta,\Om}-\lambda\,I_{\Om}))+1
\no \\
&= N_D(\lambda)
+ \dim(\ker (-\Delta_{\Theta,\Om}-\lambda\,I_{\Om}))+1.
\end{align}
Specializing this to the case when $\lambda=\lambda_{D,\Om,j}$ then yields 
\begin{align}\label{FGN-12}
\#\,\{\sigma(-\Delta_{\Theta,\Om})\cap(-\infty,\lambda_{D,\Om,j})\}
& =N_{\Theta}(\lambda_{D,\Om,j}) 
- \dim(\ker (-\Delta_{\Theta,\Om}-\lambda_{D,\Om,j}\,I_{\Om})) 
\no \\
& \geq N_D(\lambda_{D,\Om,j})+1\geq j+1.
\end{align}
Now, the fact that
$\#\,\{\sigma(-\Delta_{\Theta,\Om})\cap(-\infty,\lambda_{D,\Om,j})\} \geq j+1$ is
reinterpreted as \eqref{FGM-2}.
\end{proof}
%%%%%%%%%%%%%%%%%%%%%%%%%%

We briefly pause to describe a class of examples satisfying the hypotheses 
of Theorem \ref{t5.1A}:

%%%%%%%%%%%%%%%%%%
\begin{example} \lb{e5.1a}
Consider the special case $s=1/2$ in the compact embedding result \eqref{EQ1}. Then a  class of (generally, nonlocal) Robin boundary conditions satisfying the hypotheses 
of Theorem \ref{t5.1A} is generated by any operator 
$T \in \cB(L^2(\dOm; d^{n-1}\omega))$ satisfying $T \leq 0$ since the composition of 
$T$ with the compact embedding operator 
\begin{equation}
J_{H^{1/2}(\partial\Omega)} \colon 
H^{1/2}(\partial\Omega)\to L^2(\partial\Omega;d^{n-1} \omega)
\end{equation} 
yields a boundary operator $\Theta = T J_{H^{1/2}(\partial\Omega)} \in 
\cB_{\infty}\big(H^{1/2}(\dOm),L^{2}(\dOm)\big)$ and hence 
$\Theta \in \cB_{\infty}\big(H^{1/2}(\dOm),H^{-1/2}(\dOm)\big)$ is of the type 
$\Theta_2$ in Hypothesis \ref{h2.2}. 
\end{example} 
%%%%%%%%%%%%%%%%%%

We note that condition \eqref{FGM-1} in Theorem \ref{t5.1A} can be further refined and we will return to this issue in our final Remark \ref{r5.4}.

The case treated in \cite{Fi04} is that of a local Robin boundary condition.
That is, it was assumed that $\Theta$ is the operator of multiplication $M_{\theta}$
by a function $\theta$ defined on $\dOm$ (which satisfies appropriate
conditions). To better understand the way in which this scenario relates
to the more general case treated here, we state and prove the following result: 

%%%%%%%%%%%%%%%%%%%%
\begin{lemma}\lb{La.1}
Assume Hypothesis \ref{h2.1} and suppose that $\Theta=M_\theta$, the operator
of multiplication with a measurable function $\theta:\dOm\to\bbR$.
Suppose that $\theta\in L^p(\dOm;d^{n-1}\omega)$, where
\begin{equation}\label{Fpp}
p=n-1 \,\mbox{ if } \, n>2,  \mbox{ and } \, p\in(1,\infty] \,\mbox{ if } \, n=2. 
\end{equation}
Then
\begin{equation}\label{FGN-13}
\Theta\in\cB_{\infty}\big(H^{1/2}(\dOm),H^{-1/2}(\dOm)\big)
\end{equation}
is a self-adjoint operator which satisfies
\begin{equation}\label{FGN-14}
\|\Theta\|_{\cB(H^{1/2}(\dOm),H^{-1/2}(\dOm))}
\leq C\|\theta\|_{L^p(\dOm;d^{n-1}\omega)},
\end{equation}
where $C=C(\Om,n,p)>0$ is a finite constant.
\end{lemma}
%%%%%%%%%%%%%%%%%%%%%%%%%%%%%%%
\begin{proof}
Standard embedding results for Sobolev spaces (which continue to hold
in the case when the ambient space is the boundary of a bounded Lipschitz domain)
yield that
\begin{equation}\label{FGN-15}
H^{1/2}(\dOm)\hookrightarrow L^{q_0}(\dOm;d^{n-1}\omega), 
\, \mbox{ where }\,  q_0:=\begin{cases}
\frac{2(n-1)}{n-2} \, \mbox{ if }n>2, \\
\mbox{any number in $(1,\infty)$ if $n=2$}.
\end{cases}
\end{equation}
Since the above embedding is continuous with dense range, via duality we
also obtain that
\begin{equation}\label{FGN-16}
L^{q_1}(\dOm;d^{n-1}\omega)\hookrightarrow
H^{-1/2}(\dOm), \, \mbox{ where }\, 
q_1:=\begin{cases}
\frac{2(n-1)}{n}\mbox{ if }n>2, \\
\mbox{any number in $(1,\infty)$ if $n=2$}.
\end{cases} 
\end{equation}
Together, \eqref{FGN-15} and \eqref{FGN-16} yield that
\begin{equation}\label{FGN-17}
\cB\big(L^{q_0}(\dOm;d^{n-1}\omega),L^{q_1}(\dOm;d^{n-1}\omega)\big)
\hookrightarrow\cB\big(H^{1/2}(\dOm),H^{-1/2}(\dOm)\big),
\end{equation}
continuously. With $p$ as in the statement of the lemma, H\"older's
inequality yields  that
\begin{equation}\label{FGN-18}
M_\theta\in\cB\big(L^{q_0}(\dOm;d^{n-1}\omega),
L^{q_1}(\dOm;d^{n-1}\omega)\big)
\end{equation}
and
\begin{equation}\label{FGN-19}
\|M_\theta\|_{\cB(L^{q_0}(\dOm;d^{n-1}\omega),L^{q_1}(\dOm;d^{n-1}\omega))}
\leq C\|\theta\|_{L^p(\dOm;d^{n-1}\omega)},
\end{equation}
for some finite constant $C=C(\dOm,p,q_0,q_1)>0$, granted that
\begin{equation}\label{FGN-20}
\frac{1}{p}+\frac{1}{q_0}\leq\frac{1}{q_1}.
\end{equation}
Inequality \eqref{FGN-20} then holds with equality when $n>2$ and, given 
$p\in(1,\infty)$,
$q_0$, $q_1$ can always be chosen as in \eqref{FGN-15} and \eqref{FGN-16} when
$n=2$ so that \eqref{FGN-20} continues to hold in this case as well.
In summary, the above reasoning shows that
$\Theta=M_{\theta}\in\cB\big(H^{1/2}(\dOm),H^{-1/2}(\dOm)\big)$ and the 
estimate \eqref{FGN-14} holds. Let us also point out that $\Theta$ is a
self-adjoint operator, since $\theta$ is real-valued.

It remains to establish \eqref{FGN-13}, that is, to 
show that $\Theta$ is also a compact operator. To this end,
fix $p_0>p$ and let $\theta_j\in L^{p_0}(\dOm;d^{n-1}\omega)$, $j\in\bbN$,
be a sequence of real-valued functions with the property that $\theta_j\to\theta$
in $L^p(\dOm;d^{n-1}\omega)$ as $j\to\infty$. Set $\Theta_j:=M_{\theta_j}$,
$j\in\bbN$. From what we proved above, it follows that
\begin{equation}\label{FGN-21}
\Theta_j \to \Theta \, \mbox{ in } \,
\cB\big(H^{1/2}(\dOm),H^{-1/2}(\dOm)\big) \,\mbox{ as }\, j\to\infty,
\end{equation}
and there exists $r\in(1/2,1)$ with the property that
\begin{equation}\label{FGN-22}
\Theta_j\in\cB\big(H^{r}(\dOm),H^{-1/2}(\dOm)\big),\quad j\in\bbN.
\end{equation}
Since the embedding $H^{r}(\dOm)\hookrightarrow H^{1/2}(\dOm)$
is compact, one concludes that
\begin{equation}\label{FGN-23}
\Theta_j\in\cB_{\infty}\big(H^{1/2}(\dOm),H^{-1/2}(\dOm)\big), 
\quad j\in\bbN.
\end{equation}
Thus, \eqref{FGN-13} follows from \eqref{FGN-23} and \eqref{FGN-21}.
\end{proof}
%%%%%%%%%%%%%%%%%%%%%%%%%%%%%%%

We end by including a special case of Theorem \ref{t5.1A}
which is of independent interest. In particular, this links our
conditions on $\Theta$ with Filonov's condition
\begin{equation}\label{FGN-24}
\int_{\dOm}d^{n-1}\omega(\xi) \,\theta(\xi)\leq 0
\end{equation}
in the case where $\Theta=M_{\theta}$.

%%%%%%%%%%%%%%%%%%%%
\begin{corollary}\lb{c5.1F}
Assume Hypothesis \ref{h2.2}, where the number $\delta>0$ is taken
to be sufficiently small relative to the Lipschitz character of $\Om$
and, in addition, suppose that
\begin{equation}\label{FGM-1D}
\Theta\leq 0
\end{equation}
in the sense that $\langle f,\Theta f\rangle_{1/2} \leq 0$ for every
$f\in H^{1/2}(\partial\Omega)$. Then \eqref{FGM-2} holds.

In particular, assuming Hypothesis \ref{h2.1} and $\Theta=M_\theta$, 
with $\theta\in L^p(\dOm;d^{n-1}\omega)$, where $p$ is as in \eqref{Fpp},
is a function satisfying \eqref{FGN-24}, then \eqref{FGM-2} holds.
\end{corollary}
%%%%%%%%%%%%%%%%%%%%
\begin{proof}
The first part is directly implied by Theorem \ref{t5.1A}.
The second part is a consequence of Lemma \ref{La.1} and the conclusion
in the first part of Corollary \ref{c5.1F}, since \eqref{FGM-1D} reduces precisely to
\eqref{FGN-24} for $\Theta=M_\theta$.
\end{proof}
%%%%%%%%%%%%%%%%%%%%

%%%%%%%%%%%%%%%%%%%%
\begin{remark} \lb{r5.4} 
After submitting our manuscript to the preprint archives we received a preprint version 
of Safarov's paper \cite{Sa08} in which an abstract approach to eigenvalue counting functions and Dirichlet-to-Neumann maps was developed. His methods permit a considerable improvement of condition \eqref{FGM-1} as described in the following:  
First, we note that in order to obtain the particular inequality
\begin{equation}
\lambda_{\Theta,\Om,j+1}<\lambda_{D,\Om,j} \, \text{ for some fixed } \, j\in\bbN,  
\lb{5.33}
\end{equation}  
the proof of Theorem \ref{t5.1A} uses condition \eqref{FGM-1} 
for only one value $\eta_j \in \bbR^n$ with $|\eta_j|^2 = \lambda_{D,\Om,j}$. Unfortunately, we have no manner to determine which $\eta_j$ to choose on the sphere $|\eta| = \lambda_{D,\Om,j}^{1/2}$ such that $e^{ix\cdot\eta_j}$ does not
belong to the finite-dimensional space $U_{\lambda_{D,\Om,j}} \, \dot{+} \, 
\ker(-\Delta_{\Theta,\Om} - \lambda_{D,\Om,j} \,I_{\Om})$.

On the other hand, applying Remark\ 1.11\,(3) of Safarov \cite{Sa08} (and using that 
$\sigma_{\rm ess}(-\Delta_{\Theta,\Om})=\emptyset$), to prove that the slightly weaker inequality 
\begin{equation}
\lambda_{\Theta,\Om,j+1} \leq \lambda_{D,\Om,j} \, \text{ for some fixed } \, j\in\bbN,  
\lb{5.34}
\end{equation}  
holds, it suffices to find just one element $u_j\in H^1(\Om)\backslash H_0^1(\Om)$ satisfying
\begin{equation}
\Delta u_j \in L^2(\Om; d^n x), \quad -\Delta u_j =  \lambda_{D,\Om,j} u_j,   \lb{5.35}
\end{equation}
and
\begin{equation}
a_{\Theta} (u,v) -  \lambda_{D,\Om,j} \|u_j\|_{L^2(\Om; d^n x)}^2 \leq 0.    \lb{5.36}
\end{equation}
Since one can choose $u_j (x) = e^{i x \cdot \eta_j}$ for any $\eta_j \in\bbR^n$ with 
$|\eta_j| = \lambda_{D,\Om,j}^{1/2}$, as long as \eqref{FGM-1} holds for 
$\eta = \eta_j$, this proves that \eqref{5.34} holds whenever
\begin{equation}
\big\langle\gamma_D(e^{ix\cdot\eta_j}),
\Theta\gamma_D(e^{ix\cdot\eta_j})\big\rangle_{1/2} \leq 0   \lb{5.37}
\end{equation}
for a single vector $\eta_j\in\bbR^n$ with $|\eta_j| = \lambda_{D,\Om,j}^{1/2}$.

Going further, and applying Remark\ 1.11\,(4) of Safarov \cite{Sa08} (see also the proof of Corollary\ 1.13 in \cite{Sa08}), one obtains 
strict inequality in \eqref{5.34} if there exist two elements 
$u_{j,1}, u_{j,2} \in H^1(\Om)\backslash H_0^1(\Om)$ satisfying \eqref{5.35} 
and \eqref{5.36} and $\text{\rm lin.span} \, \{u_{j,1}, u_{j,2}\}$ does not contain an element satisfying the boundary condition in $-\Delta_{\Theta,\Om}$. But the latter follows from \eqref{FGN-6}. The two elements $u_{j,1}, u_{j,2}$ can again be chosen as
$u_{j,k} (x) = e^{i x \cdot \eta_{j,k}}$ for any $\eta_{j,k} \in\bbR^n$ with 
$|\eta_{j,k}| = \lambda_{D,\Om,j}^{1/2}$, $k=1,2$, as long as \eqref{FGM-1} holds for 
$\eta = \eta_{j,1} \text{ and } \eta_{j,2}$. Summing up,  
\begin{equation}
\lambda_{\Theta,\Om,j+1} < \lambda_{D,\Om,j} \, \text{ for some fixed } \, j\in\bbN,  
\lb{5.38}
\end{equation}  
holds whenever
\begin{equation}
\big\langle\gamma_D(e^{ix\cdot\eta_{j,k}}),
\Theta\gamma_D(e^{ix\cdot\eta_{j,k}})\big\rangle_{1/2} \leq 0   \lb{5.39}
\end{equation}
for two vectors $\eta_{j,k}\in\bbR^n$ with $|\eta_{j,k}| = \lambda_{D,\Om,j}^{1/2}$, 
$k=1,2$. 

While \eqref{5.37} as well as \eqref{5.39} assume the {\it a priori} knowledge of 
$\lambda_{D,\Om,j}$, one can finesse this dependence as follows: For instance, 
\eqref{5.34} holds for all $j\in\bbN$ whenever the set of $\eta$ satisfying inequality 
\eqref{5.37} intersects every sphere in $\bbR^n$ centered at the origin. Similarly, 
if for some $\lambda_0 > 0$, 
\begin{equation} 
\big\langle\gamma_D(e^{ix\cdot\eta_0}),\Theta 
\gamma_D(e^{ix\cdot\eta_0})\big\rangle_{1/2}
< 0 \, \text{ for some } \, \eta_0 \in\bbR^n \, \text{ with } \, |\eta_0| = \lambda_0,  \lb{5.40}
\end{equation}
then by continuity of \eqref{5.40} with respect to $\eta_0$ (using the boundedness property $\Theta\in\cB\big(H^{1/2}(\dOm),H^{-1/2}(\dOm)\big)$), one infers that \eqref{5.38} holds for all eigenvalues sufficiently close to $\lambda_0$, etc.  
\end{remark}
%%%%%%%%%%%%%%%%%%%%
\medskip

%%%%%%%%%%%%%%%%%%%%%%%%%%%%%%%%%%%%%
\noindent {\bf Acknowledgments.}
We are indebted to Mark Ashbaugh for helpful discussions and very valuable hints with regard to the literature and especially to Yuri Safarov for pointing out the validity of 
Remark \ref{r5.4} to us.
%%%%%%%%%%%%%%%%%%%%%%%%%%%%%%%%%%%%%

%%%%%%%%%%%%%%%%%%%%%%%%%%%%%%%%%%%%%%
%%%%%%%%%%%%%%%%%%%%%%%%%%%%%%%%%%%%%%

\end{document}